\documentclass{article}

\usepackage{microtype}
\usepackage{graphicx}
\usepackage{booktabs} 

\usepackage{hyperref}
\usepackage{lipsum} 
\usepackage{wrapfig}
\usepackage{subcaption}
\usepackage{comment}
\usepackage{hyperref}
\usepackage{url}
\usepackage{amsmath,amssymb,amsfonts}
\usepackage{graphicx}
\usepackage{textcomp}
\usepackage{booktabs}
\usepackage{xcolor} 
\usepackage{epsfig}
\usepackage{tabularx}
\usepackage{graphicx}
\usepackage[ruled,vlined]{algorithm2e}
\usepackage{enumitem}
\usepackage{eufrak}
\usepackage{lipsum}
\usepackage{relsize}
\usepackage{multicol}
\usepackage{multirow}
\usepackage{wrapfig}
\usepackage[hang,flushmargin]{footmisc}
\usepackage[ruled]{algorithm2e}

\usepackage{booktabs}
\usepackage{siunitx}
\usepackage{cancel}
\usepackage{xcolor}
\definecolor{darkgreen}{rgb}{0.00, 0.39, 0.00}
\definecolor{darkred}{rgb}{0.55, 0.00, 0.00}
\definecolor{darkblue}{rgb}{0.05, 0.05, 0.7}

\usepackage[final]{neurips_2024}

\usepackage{mathtools}
\usepackage{amsthm}

\usepackage[capitalize,noabbrev]{cleveref}

\theoremstyle{plain}
\newtheorem{theorem}{Theorem}[section]
\newtheorem{proposition}[theorem]{Proposition}
\newtheorem{lemma}[theorem]{Lemma}
\newtheorem{corollary}[theorem]{Corollary}
\theoremstyle{definition}
\newtheorem{definition}[theorem]{Definition}

\newtheorem{remark}[theorem]{Remark}

\usepackage[textsize=tiny]{todonotes}

\title{A Convergence-Guaranteed Algorithm for
		Stochastic Optimal Control Problems}

\author{ Mohsen Amidzadeh\\
	Department of Computer Science\\
	Aalto University\\
	\texttt{mohsen.amidzade@aalto.fi}
}

\begin{document}

\maketitle

\begin{abstract}
Stochastic Optimal Control Problems (SOCPs) plays a major role in the sequential decision-making challenges.
There exist various iterative algorithms, under framework of stochastic maximum principle, that sequentially
find the optimal control decision.
However, they are based on the adjoint sensitivity analysis that necessitates simulation of an adjoint process 
— typically a backward stochastic differential equation (SDE) that must simultaneously be adapted to a forward filtration and satisfy a terminal condition 
— which substantially increases complexity and exacerbates the curse of dimensionality.
We instead develop a stochastic maximum principle based on the Malliavin calculus,
which enables us to devise an iterative algorithm without need of an adjoint process.
Our algorithm however needs the Malliavin derivative 
that can be efficiently computed based on a forward simulator.
Empirical comparisons against standard iterative algorithms demonstrate 
that our approach alleviates the dimensionality bottleneck 
while delivering competitive performance on the considered SOCPs.

\end{abstract}

\section{Introduction}
Stochastic optimal control problems (SOCPs) stand as a vibrant research domain 
presenting extensive applications across various fields 
such as nonlinear filtering, game theory, and mathematical finance \citep{ma1999forward,zhang2017backward,Zhang2022BSDE,bouchard2018,hientzsch2019}.
These problems are characterized by functional objectives
depending on diffusion processes generated from stochastic differential equations (SDEs).

There are two main frameworks that characterize the optimal solution of SOCPs,
namely Stochastic Maximum Principle (SMP) and Hamilton-Jacobi-Bellman (HJB) equations \citep{yong1999stochastic}.
Under the framework of SMP, there are various iterative algorithms that find the optimal control decision.
\citet{Gong2017Efficient} leverages a gradient projection method to develop an iterative algorithm 
based on the adjoint sensitivity analysis.
Specifically, their approach needs simulation of an adjoint process (or a backward SDE)
that must simultaneously be adapted to a forward filtration and complies with a terminal condition.
This requirement challenges the solution methodology; the available mechanisms need extensive computations 
and suffer from the curse of dimensionality \citep{archibald2020stochastic}.
A remedy for this challenge is to approximate the adjoint process based on a time-reversed computation
\citet{archibald2020stochastic,WangLiu2025}.
Although, this approximating method is computationally effective, 
but it cannot give the exact solution and sometimes needs extra iteration updates to converge.
Conversely, our paper develops a SMP based on the \textbf{Mal}liavin calculus, 
and accordingly a \textbf{g}radient \textbf{pro}jection algorithm, termed as Mal-GPro.
In contrast to the aforementioned approaches,
our algorithm does not need an additional adjoint process, 
and more importantly is accurate from this sense that it does not use any approximation.
Our algorithm manages to replace the adjoint state with the Malliavin derivative of the original SDE 
that can be computed based on a forward simulator.

This paper thus provides the following contributions.
\begin{itemize}[leftmargin=*]
	\item We derive a stochastic maximum principle for vector (and scalar) SOCPs relying on the Malliavin calculus.
	We then derive optimality conditions for the optimal control decision.
	
	\item Based on the developed framework, we devise a gradient projection algorithm
	replacing the adjoint state with a Malliavin derivative,
	that iteratively find the optimal solution for some classes of SOCPs.
	
	\item We benchmark our algorithm against other iterative algorithms
	using numerous experiments, which shows the effectiveness of our algorithm
	from dimensionality bottleneck and performance perspectives.
\end{itemize}

\section{Related Work}
\textbf{Sensitivity analysis-based methods:}
Our work has some connections with the sensitivity analysis works that leverage the Malliavin calculus.
\citet{GobetMunos2005} performs a sensitivity analysis for stochastic objectives 
depending on the final solution of a (controlled) SDE.
\citet{Meyer2012} derives a stochastic maximum principle framework for \emph{scalar} SOCPs based on the Malliavin calculus.
In contrast, we consider more general forms of  vector SOCPs,
and more importantly, we devise an iterative projection algorithm to obtain the optimal control decision.

\textbf{Iterative gradient algorithms:}
Gradient algorithms manage to find the control decision for SOCPs based on the framework of SMP.
A gradient projection method has been developed in \citep{Gong2017Efficient} based on the adjoint backward SDE.
The adjoint state should be adapted to a forward Wiener filtration and follow the terminal condition.
Therefore, the available solutions to simulate this SDE, either mesh-based non-parametric or mesh-free parametric methods, 
suffer from the curse of dimensionality \citep{archibald2020stochastic}.
Accordingly, an approximate solution has been proposed in \citet{archibald2020stochastic,WangLiu2025}
to simulate the backward SDE based on a time-reversed computation.
However, the solution exhibits an approximation which suffers from the performance drop. 
In contrast,
our proposed algorithm replaces the backward adjoint SDE with a Malliavin derivative 
that does not require backward simulation.
More importantly, it does not rely on any approximation to simulate the Mallaivin derivative.

\section{Background}

\subsection{Variational Analysis of SOCPs}\label{Sec:variational-introduction}
In this paper, we consider the filtered probability space $(\Omega,\mathcal{F},\mathcal{P})$,
where $\Omega$ is the sample space, 
$\mathcal{F}$ the $\sigma$-algebra event space,
and $\mathcal{P}:\mathcal{F} \to [0,1]$ the probability measure.
A $l$-dimensional Wiener process $\textbf{W}=\{w_t^l\}_{l=1}^d$ with diffusion matrix $\textbf{Q}=[q_{i,j}]_{i,j}$ is then defined
with the natural filtration $\mathcal{F}_t \subseteq \mathcal{F}$ that encodes all information up to time $t$.

We now consider the following controlled  SDE:
\begin{align}\label{EQ:FBSDE}
	\!\!\!	
	\begin{cases}
		d\textbf{x}_t^\textbf{u} = \boldsymbol{a}(\textbf{x}_t^\textbf{u}, \textbf{u}_t) \: dt 
		+ \sum_{i=1}^l \boldsymbol{b}_i(\textbf{x}_t^\textbf{u}, \textbf{u}_t) \: dw_t^i, \\
		\textbf{x}_0~\text{given}
	\end{cases}	
\end{align}
where $\textbf{x}_t^\textbf{u} \in \mathbb{R}^n$ is the process, 
$\boldsymbol{a}:  \mathbb{R}^n \times U  \to \mathbb{R}^n$ is the drift function 
which describes the dynamics of the SDE and is governed by a $\mathcal{F}_t$-adapted control process $\textbf{u}_t \in U$,
with $U$ an open convex set in $\mathbb{R}^k$,
$\boldsymbol{b}_i:  \mathbb{R}^n \times U  \to \mathbb{R}^n$ is the diffusion function
determining the extent of uncertainty added to the dynamics,
and $\textbf{B}=[\textbf{b}_1,...,\textbf{b}_d]$.
Hereafter, we may use the notation $\textbf{x}_t$, instead of $\textbf{x}_t^\textbf{u}$, for the simplicity. 
A cost functional pertaining to this SDE is then considered as follows:
\begin{align}\label{EQ:functional}
	\hspace{-8 pt}	J(\textbf{u},\textbf{x}_0) = \mathbb{E}\Big\{ \int_0^T L(\textbf{x}_t^\textbf{u}, \textbf{u}_t, t)\:dt + h(\textbf{x}_T^\textbf{u}) \Big\},
\end{align}
where 
$L:\mathbb{R}^n \times U \times [0,T] \to \mathbb{R}$
and $h^f:\mathbb{R}^n \to \mathbb{R}$ are some measurable functions.
The stochastic control problem of interest is then formulated as:
\begin{align}\label{EQ:StochasticControl}
	\mbox{P}_1:~ &\max_{\textbf{u} \in U} ~ J(\textbf{u},\textbf{x}_0) \notag \\
	&\quad \text{s.t.~~~} d\textbf{x}_t = \boldsymbol{a}(\textbf{x}_t, \textbf{u}_t) \: dt 
	+ \sum_{i=1}^l \boldsymbol{b}_i(\textbf{x}_t, \textbf{u}_t) \: dw_t^i.
\end{align}
A perturbation analysis is then employed 
to find the optimal control of SOCP~$\mbox{P}_1$.
For this, the control $\textbf{u}_t\in U$ is  perturbed through $\textbf{u}_t + \epsilon \textbf{v}_t$,
and then the variation of functional~(\ref{EQ:functional}) is analyzed along direction $\textbf{v}_t$ based on the following definition.
\begin{definition}\label{Def1}
	The variation of functional $J(\textbf{u},\textbf{x}_0)$ along direction $\textbf{v}_t$ is:
	\begin{align*}
	\delta_\textbf{v}J(\textbf{u},\textbf{x}_0) = \lim_{\epsilon\to 0}\frac{1}{\epsilon}\left(J(\textbf{u}+ \epsilon \textbf{v},\textbf{x}_0) - J(\textbf{u},\textbf{x}_0)\right)
\end{align*}
\end{definition}
This definition shows how the functional $J(\textbf{u}(\cdot),\cdot)$ evolves by slightly perturbing the process $\textbf{u}(\cdot)$.
Using this definition, one can have:
\begin{proposition}\label{Propos:Variation}
	Consider SOCP $\mbox{P}_1$ with the cost functional $J(\textbf{u},\textbf{x}_0)$~(\ref{EQ:functional}),
	its variation along the direction $\textbf{v}(\cdot)$ is obtained by:
	\begin{align*}
		\delta_{\textbf{v}} J(\textbf{u},\textbf{x}_0) = \mathbb{E}\left\{ \int_0^T H_{\textbf{u}}\big(\textbf{x}, \textbf{y}, \textbf{u}, \textbf{Z}, t\big)^\top\textbf{v}_t dt \right\},
	\end{align*}
	where $ H_{\textbf{u}}$ is the gradient of the Hamiltonian $H$ w.r.t. $\textbf{u}_t$ with 
	\begin{align*}
		H\big(\textbf{x}, \textbf{y}, \textbf{u}, \textbf{Z}, t\big)  :=  
		L\big(\textbf{x}_t, \textbf{u}_t, t\big) 
		+ \boldsymbol{a}\big( \textbf{x}_t, \textbf{u}_t \big)^\top \textbf{y}(t) 
		+ \sum_{i=1}^l \boldsymbol{b}_i \big( \textbf{x}_t, \textbf{u}_t \big)^\top \textbf{z}_t^i,
	\end{align*}
	and $\textbf{y}_t \in \mathbb{R}^n$ is an adjoint process which is obtained based on the following backward SDE:
	\begin{align}\label{Eq:adjoint-process}
		\begin{cases}
			d\textbf{y}_t &= -H_{\textbf{x}}\big(\textbf{x}, \textbf{y}, \textbf{u}, \textbf{Z}, t\big)\:dt 
			+ \sum_{i=1}^l \textbf{z}_t^i dw^i_t \\
			\textbf{y}_T &= \nabla_{\textbf{x}(t)} h(\textbf{x}_T),
		\end{cases}
	\end{align}
	with $\textbf{z}^i \in \mathbb{R}^n$ being a control process associated to the adjoint process $\textbf{y}$
	\textcolor{black}{where $\textbf{z}^i_t = \mathcal{J}_{\textbf{x}(t)}\textbf{y}_t^\top \boldsymbol{b}_i(\textbf{x}_t,\textbf{u}_t)$,}
	$\textbf{Z} = \{\textbf{z}^i\}_{i=1}^l$, 
	and $H_{\textbf{x}}$ is the gradient of $H$ w.r.t. $\textbf{x}$.
\end{proposition}
\begin{proof}
	Please refer to \citep{yong1999stochastic,Geng2021} for the proof. 
\end{proof}

\begin{definition}\label{Def2}
	The control process $\textbf{u}^*$ is called a critical decision for problem $P_1$,
	if $\delta_\textbf{v} J(\textbf{u}^*,\textbf{x}_0) = 0$ through the perturbation $\textbf{u}^* + \epsilon \textbf{v}$
	for any $\mathcal{F}_t$-adapted process $\textbf{v}_t \in U$.
\end{definition}

\noindent \textbf{Iterative algorithm:}
For deterministic control decisions,
a gradient projection method can be established
to find the optimum solution $\textbf{u}$ \citep{archibald2020stochastic,WangLiu2025}.
This method utilizes the following update rule:
\begin{align}\label{Eq:gradient-projection}
\textbf{u} \leftarrow \mathsf{P}\left(\textbf{u} + \lambda\: \mathbb{E}\big\{ H_u(\textbf{x},\textbf{y},\textbf{u},\textbf{Z},t) \big\}\right),
\end{align}
where $\lambda$ is the learning rate and $\mathsf{P}(\cdot)$ is a projection operator to $U$.

\subsection{Malliavin Calculus}
\label{Sec:MalliavinDerivative}
Malliavin derivative is a tool in the stochastic analysis that
capture the sensitivity of processes with respect to the variations of Wiener process $\textbf{W}=(w_s)_{s\in[0,t]}$.
Assume process $x_t = \int_0^t g_t dw_t$ for deterministic function $g_t$, 
then its Malliavin derivative against $w_s$, denoted by $D_sx_t$, is defined to be $D_sx_t = g_s$ for all $s\leq t$.
The chain rule also holds for this derivative,
implying that for an adapted process $u_t$ and a smooth functional $F(\cdot)$ in the domain of the Malliavin derivative,
we can get $D_s F(u_t) = \frac{\partial  F(u_t)}{\partial u_t} D_su_t$.
We also need to mention the integration-by-parts formula in the Malliavin calculus \citep{alos2021malliavin,nualart2018introduction}, 
which plays a vital role in this work.
It states that for the functional $F(\cdot)$ 
and any adapted process $u_s$ with square-integrability, one can get 
\begin{align}\label{Eq:integration-by-parts}
\mathbb{E}\Big\{ F \int_0^t u_s \: dw_s  \Big\} = \mathbb{E}\Big\{ \int_0^t D_sF \:u_s \:ds \Big\}.
\end{align}
We also need to mention the Malliavin derivative of a diffusion process.
For this, consider the $n$-dimensional SDE \vspace{-5 pt}
\begin{align}\label{Eq:SDE}
	\text{SDE}(\mathbf{a}, \mathbf{B}, \mathbf{x}_0): ~~ d\textbf{x}_t = \textbf{a}(\textbf{x}_t, t) dt + \sum_{l=1}^d \textbf{b}_l(\textbf{x}_t, t) dw^l_t
\end{align}
with $\textbf{a}: \mathbb{R}^n\!\times\![0,T] \!\to\! \mathbb{R}^n $ 
and $\textbf{B}=[\textbf{b}_1,...,\textbf{b}_d]:\mathbb{R}^n\!\times\![0,T] \!\to\! \mathbb{R}^{n\times d}$ 
being measurable functions with bounded partial derivatives.
Then, 
the Malliavin derivative of $\textbf{x}_t$ against $w_s^i$ follows  \citep{nualart2018introduction}:
\begin{align}\label{Eq:MalliavinDerv0}
	D_s^i \textbf{x}_t=\textbf{b}_i(\textbf{x}_t)
	+\sum_{l=1}^d \int_s^t \mathcal{J}_{\textbf{x}}\textbf{b}_l(\textbf{x}_\tau) D_s^i \textbf{x}_\tau d w_\tau^l
	+\int_s^t \mathcal{J}_{\textbf{x}}\textbf{a}(x_\tau) D_s^i \textbf{x}_\tau d \tau,~~~i\in\{1,...,d\}.
\end{align}
Moreover, the Malliavin matrix $D_s\textbf{x}_t=[D_s^1\textbf{x}_t,\ldots,D_s^d\textbf{x}_t]$ follows \citep{nualart2018introduction}
\begin{align}\label{Eq:MatrixMalliavin}
	D_s\textbf{x}_t = \boldsymbol{\Gamma}_{s,t} \: \textbf{B}(\textbf{x}_s, s),
\end{align}
where $\boldsymbol{\Gamma}_{s,t} = \textbf{Y}_t \: \textbf{Z}_s$, 
and the matrix-valued stochastic flows $\textbf{Y}_t = \frac{\partial \textbf{x}_t}{\partial \textbf{x}_0} 
\in \mathbb{R}^{n\times n}$ and $\textbf{Z}_t = \textbf{Y}_t^{-1} \in \mathbb{R}^{n\times n}$ 
comply with the following variational SDEs:
\begin{align}\label{Eq:Y-Z}
	\textbf{Y}_t&=\textbf{I}+\sum_{l=1}^d \int_0^t \mathcal{J}_\textbf{x} \textbf{b}_l\left(\textbf{x}_s,s\right) \textbf{Y}_s d w_s^l
	+\int_0^t \mathcal{J}_\textbf{x} \textbf{a}\left(\textbf{x}_s,s\right) \textbf{Y}_s d s \notag\\
	\textbf{Z}_t &= \textbf{I}-\sum_{l=1}^d \int_0^t \textbf{Z}_s \mathcal{J}_\textbf{x} \textbf{b}_{l}\left(\textbf{x}_s,s\right) d w_s^l
	-\int_0^t \textbf{Z}_s\left(\mathcal{J}_\textbf{x} \textbf{a}\left(\textbf{x}_s,s\right)-\sum_{l=1}^d \mathcal{J}_\textbf{x} \textbf{b}_{l}\left(\textbf{x}_s,s\right)^2 \right) d s.
\end{align}
It is noteworthy that the Malliavin derivative has the following expression for the scalar case, (i.e., $n=d=1$):
\begin{align}\label{Eq:D_sx_t}
D_sx_t = b(x_s,s)\exp\bigg( \int_s^t \big( a_x(x_t,t)-\frac12 b_x(x_t,t )^2 \big) dt + \int_s^t b_x(x_t,t) dw_t \bigg),~~\text{for}~s\leq t.
\end{align}

\section{Variational Analysis of SOCPs using Malliavin Calculus}
\label{Sec:Malliavin-Variational}
Proposition~(\ref{Propos:Variation}) find the variation of the cost functional of SOCP~$\mbox{P}_1$
based on an adjoint state governed by a backward SDE.
However, solving a backward SDE should lead to a process $\textbf{y}_t$ 
that should at the same time satisfy the terminal condition $\textbf{y}_T = \nabla_{\textbf{x}(t)} h(\textbf{x}_T)$ 
and be $\mathcal{F}_t$-adapted.
This requirement makes the solution methodologies challenging \citep{chessari2023numerical}.
Instead, we intend to develop a variational framework for SOCP $\mbox{P}_1$
based on the Malliavin calculus that can provide a more efficient solution approach.

\noindent \textbf{Scalar Scenario}:

In this section, we concentrate on a scalar case of SOCP $\mbox{P}_1$, 
where $n=k=1$.
This indeed gives an intuition for the general case 
where we have a vector problem.
In this regard, we consider the following theorem.
\begin{theorem}\label{Thm1}
	Consider SOCP $\mbox{P}_1$ (\ref{EQ:StochasticControl}) with $n=k=1$,
	the variation of cost functional $J(u,x_0)$
	along direction $v$, with perturbation $u+\epsilon\: v$, can be obtained by:
	\begingroup\makeatletter\def\f@size{9}\check@mathfonts
	\begin{align}\label{Eq:variation-scalar}
	\delta_v J(u,x_0) 
	=\! \mathbb{E}\int_0^T \! 
	&
	\bigg\{
	       \frac{a_u(s) - b_x(s)b_u(s)}{b(s)} \Big( \int_s^T \! L_x(t) D_sx_t \:dt  + h_x(x_T)D_sx_T \Big) \notag \\
	   &+  \frac{b_u(s)}{b(s)} \Big( \int_s^T \! \big(L_{xx}(t) {D_sx_t} + L_{xu}(t) D_su_t \big)D_sx_t\: dt + h_{xx}(x_T) D_sx_T^2 \Big) +L_u(s) 
	 \bigg\} v_s ds,
	\end{align}
	\endgroup
	where $D_sx_t$ (respectively, $D_su_t$) is the Malliavin derivative of $x_t$ (respectively, $u_t$) against $w_s$,
	$L_x(t)$, $a_u(s)$, and $b(s)$ are short notations 
	for $L_x(x_t,u_t,t)$, $a_u(x_s,u_s)$ and $b(x_s,u_s)$, respectively,
	with $L_x$ and $L_{xx}$ denoting the first and second gradients of $L$ 
	w.r.t $x$, and
	$L_u$, $a_u$, $b_u$ showing the gradient of $L$, $a$ and $b$ w.r.t. $u$.
	\begin{proof}
		Please refer to \cref{App1}.
	\end{proof}
\end{theorem}
We need to highlight that the presence of $b(s)$ in the denominators of \cref{Eq:variation-scalar} is not problematic, 
as $\frac{D_sx_t}{b(s)}$ exists for $t\in [0,T]$.
Variation analysis in \cref{Thm1} does not rely on any adjoint process, 
in contrast to Proposition~(\ref{Propos:Variation}) which depends on adjoint state $\textbf{y}_t$ 
as well as on control process $\{\textbf{z}_t^i\}_{i=1}^d$ (see \cref{Eq:adjoint-process}).
However,  \cref{Thm1} needs computation of the Malliavin derivative of the diffusion process, i.e., $D_sx_t$,
which can be computed through an efficient and time-reversed mechanism without involving new process (see \cref{Sec:BackwardMalliavin}).

\begin{corollary}
	Consider $u^*$ be a critical point for $J(u,x_0)$ of problem $P_1$, 
	i.e., $\delta_vJ(u^*,x_0) = 0$ through the perturbation $u^* + \epsilon v$,
	then we have:
	\begin{align}
		\label{Eq:scalar-optimalitycondition}
	\mathbb{E}
	\bigg\{ &
	\frac{a_u(s) - b_x(s)b_u(s)}{b(s)} \Big( \int_s^T \! L_x(t) D_sx_t^* \:dt  + h_x(x_T^*)D_sx_T^* \Big) \notag \\
	&+  \frac{b_u(s)}{b(s)} \Big( \int_s^T \! \big( L_{xx}(t)D_sx_t^* + L_{xx}(t)D_su_t^* \big) D_sx_t^*\: dt + h_{xx}(x_T^*) {D_sx_T^*}^2 \Big) +L_u(s) ~\Big|~ \mathcal{F}_s
	\bigg\} = 0,
	\end{align}
	where $x_t^* = x_t^{u^*}$
	\begin{proof}
		It follows from this fact that \cref{Eq:variation-scalar} holds for every $\mathcal{F}_s$-adapted process $v_s$.
	\end{proof}
\end{corollary}

\noindent \textbf{Vector Scenario}:
The following theorem obtains a variation of SOCP~$\mbox{P}_1$ in the general case of vector problems 
based on the Malliavin calculus.
\begin{theorem}\label{Thm2}
	Consider SOCP $\mbox{P}_1$ (\ref{EQ:StochasticControl}),
	the variation of cost functional $J(\textbf{u},\textbf{x}_0)$
	along direction $\textbf{v}$, with perturbation $\textbf{u}+\epsilon\: \textbf{v}$, can be obtained by:
	\begingroup\makeatletter\def\f@size{9}\check@mathfonts
	\begin{align*}
		\delta_\textbf{v} J(\textbf{u},\textbf{x}_0) 
		= \mathbb{E} \!\int_0^T \! 
		\bigg\{ \!&
		\Big( \int_s^T \!  \nabla_\textbf{x}L(t)^\top \boldsymbol{\Gamma}_{s,t} \: dt+ \nabla_\textbf{x}h(\textbf{x}_T)^\top \boldsymbol{\Gamma}_{s,T} \Big) \Big( \mathcal{J}_\textbf{u}\textbf{a}(s)
		- \sum_{l,l'}^d q_{l,l'}\mathcal{J}_\textbf{x}\textbf{b}_l(s)  \mathcal{J}_\textbf{u}\textbf{b}_{l'}(s) \Big) \\
		&+ \sum_{l=1}^d 
		\Big(
		\int_s^T \!  \big({D_s^l\textbf{x}_t}^\top  \nabla_{\textbf{x}}^2L(t) + {D_s^l\textbf{u}_t}^\top  \nabla_{\textbf{u}\textbf{x}}L(t)\big) \boldsymbol{\Gamma}_{s,t} dt
		+ {D_s^l\textbf{x}_T}^\top \nabla_{\textbf{x}}^2h(\textbf{x}_T) \boldsymbol{\Gamma}_{s,T}
		\Big) 
		\:\mathcal{J}_\textbf{u}\textbf{b}_l(s) \\
		&+ \nabla_\textbf{u}L(s)^\top 
		\bigg\} \textbf{v}_s ds,
	\end{align*}
	\endgroup
	where 
	$D_s^l\textbf{x}_t$ (respectively, $D_s^l\textbf{u}_t$) is the Malliavin derivative of $\textbf{x}_t$ (respectively, $\textbf{u}_t$) 
	against $w_t^l$,
	$\textbf{Q}=[q_{i,j}]_{ij}$ is the diffusion matrix of Wiener process with $q_{l,l'}$ the diffusion coefficient between $w_t^l$ and $w_t^{l'}$,
	$\boldsymbol{\Gamma}_{s,t}$ relates to the Malliavin derivative $D_s\textbf{x}_t$ though \cref{Eq:MatrixMalliavin} 
	(i.e., $D_s\textbf{x}_t = \boldsymbol{\Gamma}_{s,t} \: \textbf{B}(\textbf{x}_s, s)$),
	and $\mathcal{J}_\textbf{x}\textbf{a}$ (respectively, $\mathcal{J}_\textbf{u}\textbf{a}$) stands for the Jacobian matrix of vector $\textbf{a}$ 
	w.r.t. $\textbf{x}$ (respectively, $\textbf{u}$).
	\begin{proof}
		Please refer to \cref{App2}.
	\end{proof}
\end{theorem}

Again, we compare the  variation analysis of \cref{Thm2} to that of Proposition~(\ref{Propos:Variation}).
The former does not depend on any adjoint process, while the latter does.
However, \cref{Thm2} need computation of stochastic flow $\boldsymbol{\Gamma}_{s,t}$ 
pertaining to the Malliavin derivative through \cref{Eq:MatrixMalliavin},
which can be computed efficiently using the approach presented in \cref{Sec:BackwardMalliavin}.

\begin{corollary}
Consider $\textbf{u}^*$ be a critical point for $J(\textbf{u},\textbf{x}_0)$  of problem $P_1$, 
i.e., $\delta_\textbf{v}J(\textbf{u}^*,\textbf{x}_0) = 0$ through the perturbation $\textbf{u}^* + \epsilon \textbf{v}$,
we then have:
\begingroup\makeatletter\def\f@size{9}\check@mathfonts
\begin{align}\label{Eq:vector-optimalitycondition}
\mathbb{E} 
\bigg\{ \!
\Big( \int_s^T& \!  \nabla_\textbf{x}L(t)^\top \boldsymbol{\Gamma}_{s,t}^* \: dt+ \nabla_\textbf{x}h(\textbf{x}_T^*)^\top \boldsymbol{\Gamma}_{s,T}^* \Big) \Big( \mathcal{J}_\textbf{u}\textbf{a}(s)
- \sum_{l,l'}^d q_{l,l'}\mathcal{J}_\textbf{x}\textbf{b}_l(s)  \mathcal{J}_\textbf{u}\textbf{b}_{l'}(s) \Big) \notag \\
&+ \sum_{l=1}^d 
\Big(
\int_s^T \!  \big({D_s^l\textbf{x}_t^*}^\top  \nabla_{\textbf{x}}^2L(t) + {D_s^l\textbf{u}_t^*}^\top  \nabla_{\textbf{u}\textbf{x}}L(t)\big) \boldsymbol{\Gamma}_{s,t}^* dt
+ {D_s^l\textbf{x}_T^*}^\top \nabla_{\textbf{x}}^2h(\textbf{x}_T^*) \boldsymbol{\Gamma}_{s,T}^*
\Big) 
\:\mathcal{J}_\textbf{u}\textbf{b}_l(s) \notag\\
&+ \nabla_\textbf{u}L(s)^\top ~\Big|~ \mathcal{F}_s \bigg\} = 0,
\end{align}
\endgroup
where $\textbf{x}_t^* = \textbf{x}_t^{\textbf{u}^*}$ and $D_s\textbf{x}_t^* = \boldsymbol{\Gamma}_{s,t}^* \: \textbf{B}(\textbf{x}_s^*, s)$.
\end{corollary}

\section{An Iterative Approach for SOCPs based on Malliavin Calculus}
In this section, we intend to develop a gradient projection algorithm to solve SOCPs with
deterministic control decision $\textbf{u} \in \mathcal{U}$, 
where $\mathcal{U}$ here is a nonempty convex set whose elements are square-integrable.
For this purpose, we first find an expression
for the Gateaux derivative~\citep{Gong2017Efficient} of cost functional $\textbf{J}(\textbf{u},\textbf{x}_0)$,
and then outline a fixed-point scheme for optimizing the control decision $\textbf{u}$.
Accordingly, we manage to devise an iterative algorithm to solve  SOCP $\mbox{P}_1$.

We now need to draw the following remark.
\begin{remark}\label{Remark1}
	For SOCP $\mbox{P}_1$~(\ref{EQ:StochasticControl}) with deterministic control decision $\textbf{u} \in \mathcal{U}$,
	\cref{Thm1} (respectively, \cref{Thm2}) provides a representation for the Gateaux derivative of cost functional $J(u,x_0)$ (respectively, $J(\textbf{u},\textbf{x}_0)$) 
	with respect to the control decision at $t=s$ \citep{Gong2017Efficient}:
	\begin{align}\label{Eq:scalar-gradient}
		J'(u,x_0)_{|s} = 
		\mathbb{E}
		\bigg\{ &
		\frac{a_u(s) - b_x(s)b_u(s)}{b(s)} \Big( \int_s^T \! L_x(t) D_sx_t \:dt  + h_x(x_T)D_sx_T \Big) \notag \\
		&+  \frac{b_u(s)}{b(s)} \Big( \int_s^T \!L_{xx}(t) {D_sx_t}^2\: dt + h_{xx}(x_T) D_sx_T^2 \Big) +L_u(s) 
		\bigg\}
	\end{align}
	and 
	\begingroup\makeatletter\def\f@size{9}\check@mathfonts
	\begin{align}\label{Eq:vector-gradient}
		\nabla_\textbf{u}J(\textbf{u},\textbf{x}_0)_{|s} =
		\mathbb{E}
		\bigg\{ \!&
		\Big( \int_s^T \!  \nabla_\textbf{x}L(t)^\top \boldsymbol{\Gamma}_{s,t} \: dt+ \nabla_\textbf{x}h(\textbf{x}_T)^\top \boldsymbol{\Gamma}_{s,T} \Big) \Big( \mathcal{J}_\textbf{u}\textbf{a}(s)
		- \sum_{l,l'}^d q_{l,l'}\mathcal{J}_\textbf{x}\textbf{b}_l(s)  \mathcal{J}_\textbf{u}\textbf{b}_{l'}(s) \Big) \notag \\
		&+ \! \sum_{l=1}^d \Big( \int_s^T [\boldsymbol{\Gamma}_{s,t} \textbf{B}(s)]_l^\top  \nabla_{\textbf{x}}^2L(t) \boldsymbol{\Gamma}_{s,t} dt
		+ [\boldsymbol{\Gamma}_{s,T} \textbf{B}(s)]_l^\top \nabla_{\textbf{x}}^2h(\textbf{x}_T) \boldsymbol{\Gamma}_{s,T}
		\Big) \:\mathcal{J}_\textbf{u}\textbf{b}_l(s) \notag \\
		&+ \nabla_\textbf{u}L(s)^\top 
		\bigg\}.
	\end{align}
	\endgroup
\end{remark}

Having obtained the Gateaux derivative based on Remark~\ref{Remark1}, we can outline a gradient projection method inspired by the scheme~(\ref{Eq:gradient-projection}).
Specifically, for scalar SOCPs, we set the following update rule:
\begin{align}\label{Eq:gradient-projection-scalar}
	u^{i+1} \leftarrow \mathsf{P}\left(u^i + \lambda\: J'(u^i,x_0) \right),
\end{align}
and for vector SOCPs:
\begin{align}\label{Eq:gradient-projection-vector}
	\textbf{u}^{i+1} \leftarrow \mathsf{P}\left(\textbf{u}^{i} + \lambda\: \nabla_\textbf{u}J(\textbf{u}^{i},\textbf{x}_0) \right),
\end{align}
where $i$ is the iteration index.

In the sequel, we aim to numerically implement the gradient projection iteration in \cref{Eq:gradient-projection-scalar,Eq:gradient-projection-vector}.
To accomplish this goal,  
we need to 
(i) apply a discretization method on the gradient projection,
(ii) compute the stochastic flow $\boldsymbol{\Gamma}_{s,t}$ and Malliavin derivative $D_s\textbf{x}_t$ in \cref{Eq:scalar-gradient,Eq:vector-gradient},
and (iii) accordingly develop a numerical iterative algorithm to solve SOCP $\mbox{P}_1$.

\subsection{Numerical Implementation of Gradient Projection}
\label{Sec:numerical}
We here discretize the gradient projection scheme.
Accordingly, we first partition the interval $[0,T]$
based on $N$ fixed sub-intervals $\Pi_N := \{[t_{j-1},t_j)\}_{j=1}^N$,
where $t_0=0$, $t_N=T$ and  $\Delta t := t_j - t_{j-1} = \frac{T}{N}$ is the duration of each sub-interval.
Accordingly, the piece-wise control decision can be obtained by $\textbf{u}_{t_j}$ for $j\in\{1,...,N\}$.
We then apply this discretization policy for the gradient projection scheme (\ref{Eq:gradient-projection-scalar})-(\ref{Eq:gradient-projection-vector}) to get:
\begin{align}\label{Eq:gradient-projection-discrete}
	u^{i+1}_{t_j} &\leftarrow \mathsf{P}_N\left(u^i_{t_j} + \lambda\: J'^{N}(u^i,x_0)_{|t_{j}} \right),\notag \\
	\textbf{u}^{i+1}_{t_j} &\leftarrow \mathsf{P}_N\left(\textbf{u}^{i}_{t_j} + \lambda\: \nabla_\textbf{u}J^{N}(\textbf{u}^{i},\textbf{x}_0)_{|t_{j}} \right),\quad\text{for}~j\in\{1,...,N\}
\end{align}
where $\textbf{u}^{i}_{t_j}$ is the discrete approximation of $\textbf{u}_{t_j}$  at iteration $i$, $J'^N$ (respectively, $\nabla_\textbf{u}J^N$) is the discrete approximation of $J'$ (respectively, $\nabla_\textbf{u}J$),
and $\mathsf{P}_N(\cdot)$ is an operator that projects onto the set $\mathcal{U}_N$ with
\begingroup\makeatletter\def\f@size{9}\check@mathfonts
\begin{align*}
\mathcal{U}_N = \left\{ \textbf{u}\in\mathcal{U} ~\Big|~ \textbf{u} = \sum_{j=1}^N \boldsymbol{\alpha}_jI_{[t_{j-1},t_j)},~\boldsymbol{\alpha}_j \in \mathbb{R}^k \right\}.
\end{align*}
\endgroup
The following theorem then studies the convergence of the
gradient projection scheme~(\ref{Eq:gradient-projection-discrete});
Note that we present this theorem merely for the vector SOCPs, the scalar version can be simply derived.
\begin{theorem}
	Denote $\textbf{u}^{*,N}$ as the fixed-point solution of the projection scheme~(\ref{Eq:gradient-projection-discrete}), i.e.,
	$
	\textbf{u}^{*,N} = \mathsf{P}_N(\textbf{u}^{*,N} + \lambda \nabla_\textbf{u}J^N(\textbf{u}^{*,N},\textbf{x}_0))
	$.
	Further, assume $\nabla_\textbf{u}J(\textbf{u},\textbf{x}_0)$ is Lipschitz continuous with the constant $L$ around $\textbf{u}^*$ and $\textbf{u}^{*,N}$, i.e., 
		\begin{align*}
		\| \nabla_\textbf{u}(\textbf{u}^*,\textbf{x}_0) - \nabla_\textbf{u}(\textbf{u},\textbf{x}_0)\| &\leq L \| \textbf{u}^* - \textbf{u}\|,\quad\forall~u\in\mathcal{U},\\
		\| \nabla_\textbf{u}(\textbf{u}^{*,N},\textbf{x}_0) - \nabla_\textbf{u}(\textbf{u},\textbf{x}_0)\| &\leq L \| \textbf{u}^{*,N} - \textbf{u}\|,\quad\forall~u\in\mathcal{U}_N
	\end{align*}
	and uniformly monotone with the rate $r$ around $\textbf{u}^*$ and $\textbf{u}^{*,N}$, i.e., 
	\begin{align*}
	\big(\nabla_\textbf{u}(\textbf{u}^*,\textbf{x}_0) - \nabla_\textbf{u}(\textbf{u},\textbf{x}_0)\big)^\top( \textbf{u}^{*} - \textbf{u}) &\geq r\| \textbf{u}^{*} - \textbf{u}\|^2,\quad\forall~u\in\mathcal{U},\\
	\big(\nabla_\textbf{u}(\textbf{u}^{*,N},\textbf{x}_0) - \nabla_\textbf{u}(\textbf{u},\textbf{x}_0)\big)^\top( \textbf{u}^{*,N} - \textbf{u}) &\geq r\| \textbf{u}^{*,N} - \textbf{u}\|^2,\quad\forall~u\in\mathcal{U}_N,
	\end{align*}
	Moreover, $\nabla_\textbf{u}J^N$ 
	tends to
	$\nabla_\textbf{u}J$ as 
	$N \to \infty$, and $\lambda$ is chosen such 
	that $0<1-2r\lambda+(1+2L)\lambda^2<\delta^2$ 
	with $\delta\in(0,1)$.
	Then, the discrete approximation of control decision complies with the following bound as $i\to\infty$:
	$$
	\| \textbf{u}^* - \textbf{u}^i\| \sim \mathcal{O}(\Delta t).
	$$
	\begin{proof}
		It can be proved by following the procedure of Theorem 3.1 in \citep{Gong2017Efficient}.
	\end{proof}
\end{theorem}
The next section explains how the stochastic flow and Malliavin derivative required for the Gateaux derivative $\nabla_\textbf{u}J(\textbf{u},\textbf{x}_0)$ can be computed.

\begin{algorithm}[t]
	\small
	\caption{\scalebox{0.98}{Computation of Stochastic Flow $\boldsymbol{\Gamma}_{s,t}$ and Malliavin Derivative $D_s\textbf{x}_t$}}
	\label{Alg:backward-malliavin}
	\DontPrintSemicolon
	\KwData{Forward trajectory $(\textbf{x}_r, \textbf{w}_r)$, time‐step $\Delta t$.}
	\textcolor{darkblue}{\tcp{\small Initialize via terminal condition:}}
	$\boldsymbol{\Gamma}_{r,r} \gets \textbf{I}$\;
	\textcolor{darkblue}{\tcp{ Solve the SDE~(\ref{Eq:BSDE_gamma}):}}
	\For{$t \gets r$ \KwTo $T$ with step $\Delta t$}{
		Compute 
		\scalebox{0.9}{$
			\displaystyle \Delta \boldsymbol{\Gamma}_{r,t}
			\;=\;\,\Bigl( \mathcal{J}_{\textbf{x}} \textbf{a}(\textbf{x}_t,\textbf{u}_t) \Delta t
			+\sum_{l=1}^d \mathcal{J}_{\textbf{x}} \textbf{b}_l(\textbf{x}_t,\textbf{u}_t)\bigl(\textbf{w}_{t}-\textbf{w}_{t-\Delta t}\bigr)\Bigr) \boldsymbol{\Gamma}_{r,t}
			$}\;
		Update 
		$\boldsymbol{\Gamma}_{\,r\,,\,t+\Delta t}
		\;\gets\;\boldsymbol{\Gamma}_{r,t} + \Delta \boldsymbol{\Gamma}_{r,t}$\;
	}
	\textcolor{darkblue}{\tcp{ Recover Malliavin derivatives:}}
	\For{$t\in\{r,r+\Delta t,\dots,T\}$}
	{
		$D_r \textbf{x}_t \;\gets\;\boldsymbol{\Gamma}_{r,t}\;\textbf{B}\bigl(\textbf{x}_r,r\bigr)$\;
	}
	\textbf{return} $\boldsymbol{\Gamma}_{r,t}$ and $D_r\textbf{x}_t$ for $t\in\{ r,r+\Delta t,...,T\}$
	\small
\end{algorithm}

\subsection{Computation of Stochastic Flow and Malliavin Derivative}\label{Sec:BackwardMalliavin}
The Gateaux derivative in Remark~\ref{Remark1} depends on the stochastic flow $\boldsymbol{\Gamma}_{s,t}$ (or Mallaivin derivative $D_sx_t$).
Here, we intend to leverage an approach to compute this information.
In this regard, we need to first present the following proposition.
	\begin{proposition}
	\label{Propos:BackwardSDE}
	Consider the Malliavin derivative $D_r \textbf{x}_t=[D_r^1 \textbf{x}_t,...,D_r^d \textbf{x}_t]$~(\ref{Eq:MalliavinDerv0}) of SDE~(\ref{EQ:FBSDE}), 
	we have 
	\begin{align}\label{Eq:vector_malliavin}
		D_r \textbf{x}_t = \boldsymbol{\Gamma}_{r,t} \:\textbf{B}(\textbf{x}_r,r),\qquad\text{for }r\leq t,
	\end{align}
	where 
	$\boldsymbol{\Gamma}_{r,t} \in \mathbb{R}^{n\times n}$ represents a linear SDE, 
	and adapted to the filtration $\mathcal{F}_{r}$ obtained by:
	\begin{align}\label{Eq:BSDE_gamma}
		&{\rm SDE}(\mathcal{J}_{\textbf{x}}\textbf{a},\mathcal{J}_{\textbf{x}}\textbf{B},\boldsymbol{\Gamma}_{r,r}=\textbf{I}):~ 
		d \boldsymbol{\Gamma}_{r,t} = \Bigg( \sum_{l=1}^d \mathcal{J}_{\textbf{x}}\textbf{b}_l(\textbf{x}_t,\textbf{u}_t)\: dw_t^l 
		+ \mathcal{J}_{\textbf{x}}\textbf{a}(\textbf{x}_t,\textbf{u}_t) dt  \Bigg)\boldsymbol{\Gamma}_{r,t},
	\end{align}
	with initial condition $\boldsymbol{\Gamma}_{r,r}=\textbf{I}$
	\begin{proof}
		see \cref{App3}.
	\end{proof}
\end{proposition}

Proposition~\ref{Propos:BackwardSDE} together with an SDE-simulator provide a method to recursively compute the Malliavin derivative $D_r\textbf{x}_t$. Specifically, for each $t \in [0,T]$, this approach starts from the  initial condition $\boldsymbol{\Gamma}_{r,r}= \textbf{I}$, gradually increments the parameter $t$, and computes the right hand side of \cref{Eq:BSDE_gamma} 
\cref{Alg:backward-malliavin} shows pseudo-code based on the Euler-Maruyama scheme with a convergence rate proportional to  $\sqrt{\Delta t}$ \citep{kloeden1992numerical}.

\subsection{A Gradient Projection Algorithm via Malliavin Calculus}
The materials of \cref{Sec:numerical,Sec:BackwardMalliavin} enables us to present an iterative mechanism
for solving SOCP~$\mbox{P}_1$.
We accordingly devise  a projection algorithm
based on the numerical scheme in \cref{Eq:gradient-projection-discrete}, 
\cref{Remark1} and \cref{Alg:backward-malliavin}.
As this algorithm is based on a \textbf{g}radient \textbf{pro}jection approach utilizing \textbf{Mal}liavin calculus, we term it as Mal-GPro.
This algorithm includes 3 main phases as follows.
\begin{itemize}[leftmargin=*]
	\item \textbf{Forward Pass}: SDE~(\ref{Eq:SDE}) is simulated within discrete partition $\Pi_N$ using any SDE-solver, e.g. Euler-Maruyama one, to obtain the stochastic paths $(\textbf{x}_t,\textbf{W}_t)_{t\in\Pi_N}$. Additionally, we compute the stochastic flow and Malliavin derivatives $(\boldsymbol{\Gamma}_{s,t}, D_s\textbf{x}_t)_{t\in\Pi_N,s\in\Pi_N}$ with $s\leq t$ using \cref{Alg:backward-malliavin}.
	
	\item \textbf{Gradient Computation}: the discrete gradient of cost functional $\nabla_\textbf{u}J^N(\textbf{u},\textbf{x}_0)$\,/\,$J'^N(u,\textbf{x}_0)$ is computed 
	according to Remarak~\ref{Remark1} and using \cref{Eq:scalar-gradient,Eq:vector-gradient}.
	
	\item \textbf{Update Control}: the control process is then updated via the projection gradient approach \cref{Eq:gradient-projection-discrete} with
	learning rate $\lambda$.
	
The pseudo-code of this algorithm is presented in \cref{Alg:main-algorithm}.

\end{itemize}

\begin{algorithm}[t]
	\caption{Gradient Projection Algorithm based on Malliavin Calculus (Mal-GPro)}
	\label{Alg:main-algorithm}
	\KwIn{
		Initial state $\textbf{x}_0$,
	}
	\KwOut{Optimal control process $\textbf{u}$.\\}
	
	\For{$\text{i} = 1$ \KwTo $N_{\text{eps}}$}{
		\textcolor{darkblue}{\tcp{Simulate the  SDE~(\ref{Eq:SDE}) within the partition $\Pi_N$ via any SDE solver}}
		$(\mathbf{x}_t, \mathbf{w}_t) \leftarrow \texttt{SDEsolve}(\mathbf{a}, \mathbf{B}, \mathbf{x}_0)$,~$t\in \Pi_N$\;
		
		\textcolor{darkblue}{\tcp{Compute Stochastic flow $\boldsymbol{\Gamma}_{s,t}$ and Malliavin derivative $D_sx_t$ 
		by solving SDE (\ref{Eq:BSDE_gamma}) using \cref{Alg:backward-malliavin}:}}
		$(\boldsymbol{\Gamma}_{s,t}, D_s \mathbf{x}_t) \leftarrow \texttt{SDEsolve}(\mathcal{J}_{\textbf{x}}\mathbf{a}, \mathcal{J}_{\textbf{x}}\mathbf{B}, \boldsymbol{\Gamma}_{t,t}),~~s \in \Pi_N, t\in \Pi_N, s\leq t$\;
		
		\textcolor{darkblue}{\tcp{Compute the functional gradient based on the SOCP problem:}}
		
		\If{scalar}{compute $J'^N(u,x_0)_{|t}$ via \cref{Eq:scalar-gradient}.}
		\Else{compute $\nabla_\textbf{u}J^N(\textbf{u},\textbf{x}_0)_{|t}$ via \cref{Eq:vector-gradient}.}
		
		\textcolor{darkblue}{\tcp{Update control process via the projection gradient~(\ref{Eq:gradient-projection-discrete}):}}
		\If{scalar}{
			$
			u_t^{i+1} \leftarrow \mathsf{P}_N(u^i_t + \lambda J'^N(u^i,x_0)_{|t})
			$, for $t\in \Pi_N$.
		}
		\Else
		{
		$\boldsymbol{u}_t^{i+1} \leftarrow \mathsf{P}_N(\boldsymbol{u}_t^i + \lambda \nabla_\textbf{u}J^N(\textbf{u}^i,\textbf{x}_0)_{|t}  )$, for $t\in \Pi_N$.
		}
	}
\end{algorithm}

\section{State-of-the-Art Algorithms from Literature}
We compare the performance of Mal-GPro algorithm with two adjoint-based algorithms derived 
based on the stochastic maximum principle with adjoint sensitivity mechanism. 
They are explained as follows.

\subsection{Gradient Projection Method}\label{Sec:Ad-GPro}
This algorithm is established based on the variational analysis presented in \cref{Sec:variational-introduction},
and follows the update mechanism in \cref{Eq:gradient-projection} \citep{Gong2017Efficient}.
In this regard, two $\mathcal{F}_t$ adapted processes, 
i.e., adjoint process $\textbf{y}_t$ and control process $\{\textbf{z}_t^i\}_{i=1}^l$,
should be found so as to simultaneously comply with the backward SDE and follow the terminal condition in ~(\ref{Eq:adjoint-process}).
There are numerous mechanisms to find such processes, including mesh-based approaches and parameterization methods \citep{chessari2023numerical, Yong2023FBSDE}.
The former considers pre-determined grid meshes for the process $\textbf{x}_t$,
and then obtains the processes $(\textbf{y}_t, \{\textbf{z}_t^i\}_{i=1}^l)$ as a function of considered meshes.
However, this methodology needs extensive computations and suffers from the curse of dimensionality \citep{archibald2020stochastic}.
In the parameterization methods, the processes $(\textbf{y}_t, \{\textbf{z}_t^i\}_{i=1}^l)$ are parameterized, e.g. using a neural network,
such that the adjoint equation \cref{Eq:adjoint-process} is fulfilled \citep{chessari2023numerical}.
For our simulations, we leverage the parameterization methodology,
and we term it as \textbf{Ad}joint  \textbf{G}radient \textbf{Pro}jection (Ad-GPro).
With \emph{ideal} parameterization, this method obtains the exact solution \cite{Ji2022,chessari2023numerical} and thus provides a benchmark for the comparison.

\subsection{Stochastic Gradient Descent Method}
This algorithm, which provides an computationally efficient approach, has been built on stochastic procedure based on \cref{Eq:gradient-projection}.
In contrast to Ad-GPro, it approximates the process $(\textbf{y}_t, \{\textbf{z}_t^i\}_{i=1}^l)$ using a single-sample approximation 
and through a backward simulation mechanism \citep{archibald2020stochastic,WangLiu2025}.
This methodology thus reduces the extensive computations of Ad-GPro at the cost of performance decrease 
or requiring  more update iterations for convergence.
We term this algorithm  as \textbf{Ad}joint  \textbf{S}tochastic \textbf{G}radient \textbf{Descent} (Ad-SGD).

\section{Experiments}
We here consider scalar and vector experiments of stochastic optimal control problems 
for the evaluation purposes.
We then benchmark our proposed algorithm Mal-GPro against Adjoint sensitivity based algorithms, i.e., Ad-SGD and Ad-GPro.

\subsection{Scalar SOCPs}
\noindent \textbf{Algorithms Setup}:
For all of the considered algorithms, 
we set the number of batches to $100$,
and learning rate to  $\lambda = 10^{-2}$, unless specified.
For the Ad-GPro, we use a neural network parameterization
to estimate the adjoint\,/\,control processes $(\textbf{y}_t,\{\textbf{z}_t^i\}_{i=1}^l)$.
For this, we apply 3 hidden layers with 16 neurons and \emph{tanh} activation function.

\noindent\textbf{Experiment 1:}
We now consider a Black-Scholes type control problem \cite{du2013effective}
\begingroup\makeatletter\def\f@size{8.5}\check@mathfonts
\begin{align}\label{Eq:scalar-exp1}
	&\min_{u\in\mathcal{U}} J(u,x_0) = \frac{1}{2} \int_0^T \mathbb{E}[(x_t - x_t^\star)^2] dt + \frac{1}{2} \int_0^T u_t^2 dt, \notag \\
	&\text{s.t.} \quad dx_t = u_tx_tdt + \sigma x_tdw_t, 
\end{align}
\endgroup
with $x_0=1$, $T=1$, 
$x^*_t=\dfrac{e^{\sigma^2 t}-(T-t)^2}{1/{x_0}-T t+\frac{t^2}{2}}+1$,
and $\sigma=0.01$ as the diffusion coefficient.
The following analytical solution then can be found as a benchmark:
$
u_t^a = \frac{T - t}{{1}/{x_0} - Tt + \frac{t^2}{2}}, \quad
\label{eq:bs-exact-control}
$

\begin{figure}[t!]
	\centering
	\subcaptionbox{\label{Fig:scalar-exp1-objective}}
	{
		\includegraphics[width=0.35\textwidth]{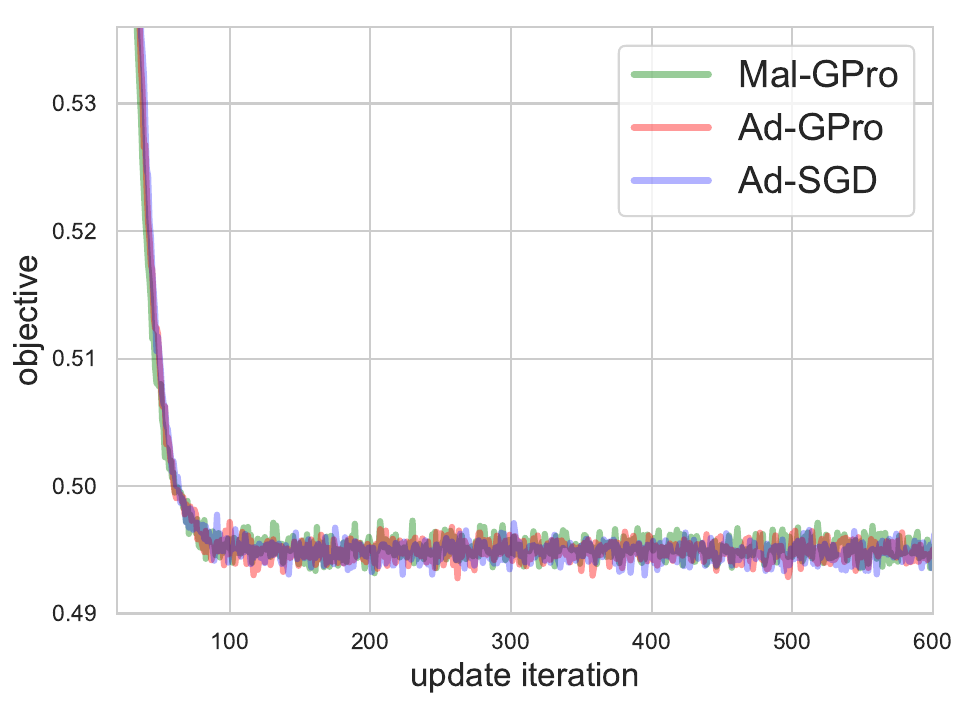}
	}
	\hspace{12 pt}
	\subcaptionbox{\label{Fig:scalar-exp1-control}}
	{
		\includegraphics[width=0.35\textwidth]{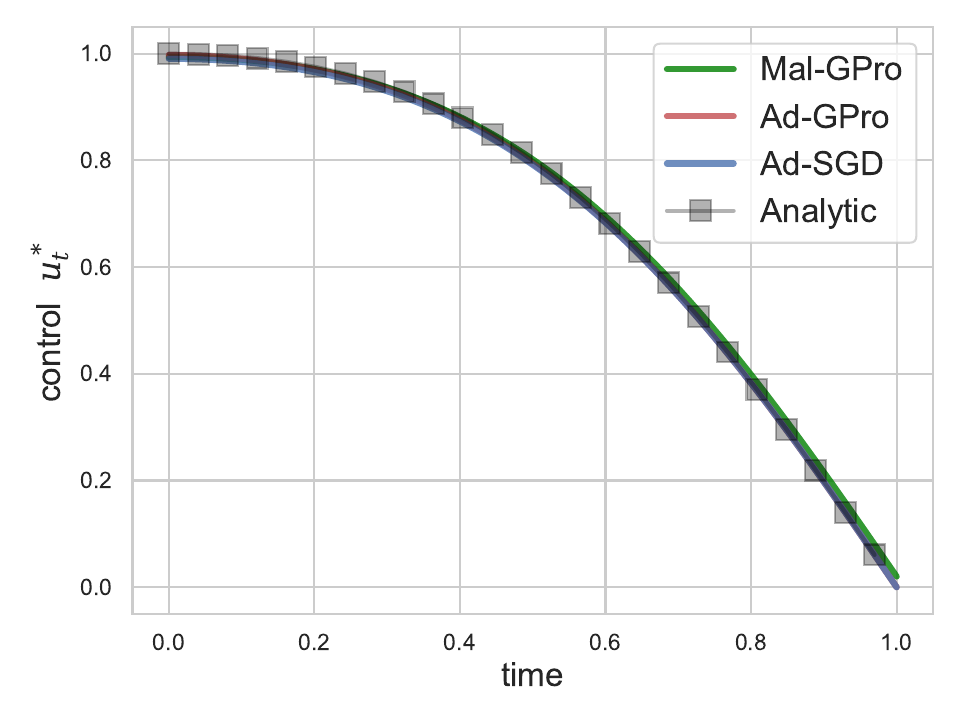}
	}
	\subcaptionbox{\label{Fig:scalar-exp1-error}}
	{
		\begin{tabular}{l l}
			\toprule
			Model & Control Error $\mathcal{E}_c$\\
			\cmidrule(lr){1-1}\cmidrule(lr){2-2}
			Ad-GPro,  \scalebox{0.8}{($\Delta t = 10^{-2}$)}	 			 &  $2.2 \times 10^{-5} \pm 2 \times 10^{-8}$ \\
			Ad-SGD,   \scalebox{0.8}{($\Delta t = 10^{-2}$)} 		 		 &  $5.5 \times 10^{-5} \pm 7 \times 10^{-7}$ \vspace{3pt}\\ 
			Mal-GPro, \scalebox{0.8}{($\Delta t = 10^{-2}$)} 				 &  $8.9 \times 10^{-5} \pm 1 \times 10^{-7}$  \\ 
			Mal-GPro, \scalebox{0.8}{($\Delta t = 5\!\times\!10^{-3}$)}  	 &  $1.7 \times 10^{-5} \pm 6 \times 10^{-8}$ \\
			\bottomrule
		\end{tabular}
	}
	\caption{
		The results of our Mal-GPro algorithm compared to the baseline algorithms, Ad-SGD and Ad-GPro, 
		for problem (\ref{Eq:scalar-exp1}).	
		(a) shows the optimum value of objective, 
		and (b) shows the optimum control decision. 
		Clearly, all of the considered algorithm can properly learn the control decision
		and their respective solutions match with the analytical one.
		(c) tabulates the control error $\mathcal{E}_c$ 
		(integral of L2-norm of difference between the obtained control and the analytical one) 
		of different schemes; Mal-GPro algorithm needs smaller step-size to reach a level of error comparable to  Ad-SGD and Ad-GPro.
	}
	\label{Fig:Exp1}
\end{figure}
\cref{Fig:Exp1} shows the results of our Mal-GPro algorithm compared to the baseline algorithms, Ad-SGD and Ad-GPro, 
for this problem.
Based on \cref{Fig:scalar-exp1-objective}, all of the considered algorithms
show the same behavior from the perspectives of performance and sample-efficiency (the update iterations it takes to converge).
Moreover, all of the algorithms can properly learn the control decision
and their respective solutions match with the analytical solution.
We also evaluate the control error of different schemes defined as below.
\begingroup\makeatletter\def\f@size{8.5}\check@mathfonts 
\begin{align}\label{Eq:control-error}
	\mathcal{E}_c = \int_0^T \| \textbf{u}_t^*  - \textbf{u}_t^{\rm a} \|^2 dt,
\end{align}
\endgroup
where $\textbf{u}_t^*$ is the optimal control obtained by considered algorithms (Mal-GPro, Ad-GPro and Ad-SGD), 
and $\textbf{u}_t^{\rm a}$ is the analytical control decision.
Table~\ref{Fig:scalar-exp1-error} shows the control error $\mathcal{E}_c$ of different schemes. 
Mal-GPro algorithm needs smaller step-size to reach an extent of control error comparable to  Ad-SGD and Ad-GPro.

\noindent\textbf{Experiment 2:}
We now consider  the following optimal control problem  \cite{du2013effective}.
\begingroup\makeatletter\def\f@size{8.5}\check@mathfonts 
\begin{align}\label{Eq:scalar-exp2}
	&\min_{u\in\mathcal{U}} J(u,x_0) = \frac{1}{2} \int_0^T \mathbb{E}[(x_t - 1)^2] dt + \frac{1}{2} \int_0^T u_t^2 dt, \notag \\
	&\text{s.t.} \quad dx_t = u_tx_tdt + \sigma \sqrt{1+x_t^2}dw_t, 
\end{align}
\endgroup
with given $x_0=1$, $T=1$ and $\sigma=0.5$ as the diffusion coefficient.
However, the analytical solution of this problem is not known.

\begin{figure}[t!]
	\centering
	\subcaptionbox{\label{Fig:scalar-exp2-objective}}
	{
		\includegraphics[width=0.35\textwidth]{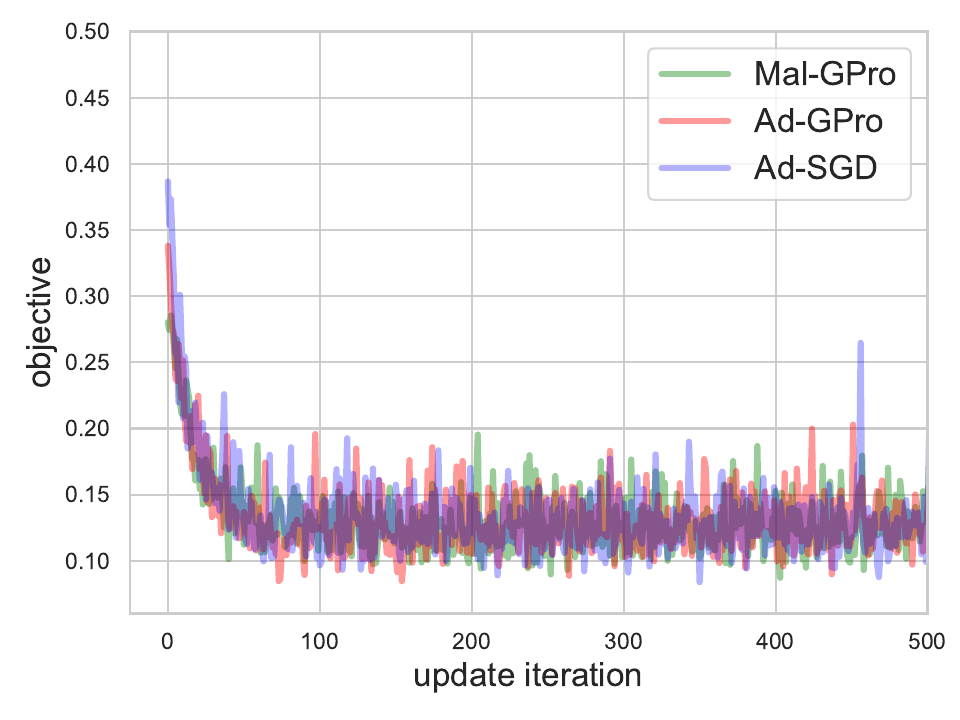}
	}
	\hspace{12 pt}
	\subcaptionbox{\label{Fig:scalar-exp2-control}}
	{
		\includegraphics[width=0.35\textwidth]{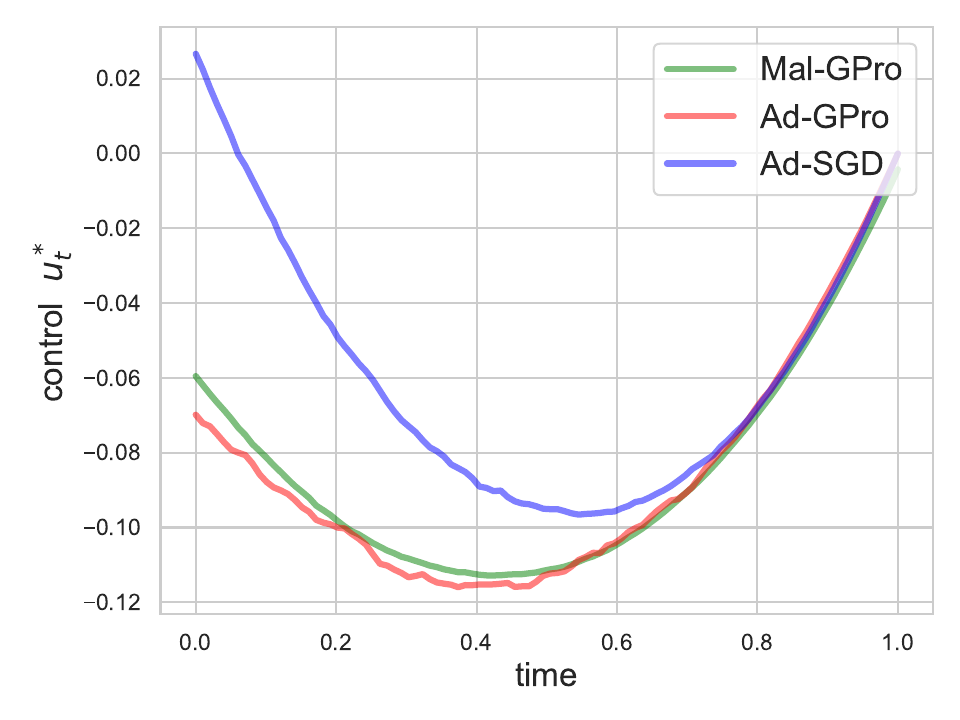}
	}
	\caption{The results of our Mal-GPro algorithm compared to the baseline algorithms, Ad-SGD and Ad-GPro, 
		for problem (\ref{Eq:scalar-exp2}).	
		(a) shows the optimum value of objective and
		(b) the optimum control decision. 
		While Ad-SGD cannot provide a proper solution, 
		the control trajectories of Mal-GPro and Ad-GPro approximately coincide.
		The control trajectory produced by Ad-GPro appears non-smooth
		due to the inherent limited capacity of the parametric approximation used in simulating the adjoint process.
	}
	\label{Fig:Exp2}
\end{figure}
\cref{Fig:Exp2} shows the results of Mal-GPro algorithm compared to  Ad-SGD and Ad-GPro algorithms, 
for this problem.
According to \cref{Fig:scalar-exp2-control}, the control trajectories of Mal-GPro and Ad-GPro approximately coincide,
while the solution of Ad-SGD deviates from them.
The reason is that Ad-SGD provides an approximation mechanism 
and it is not guaranteed to obtain the optimal solution \cite{archibald2020stochastic}.
Note that, the control trajectory produced by Ad-GPro appears non-smooth
due to the inherent limited capacity of the parametric approximation used in simulating the adjoint process.
These results portrays the effectiveness of Mal-GPro for solving SOCPs.

\subsection{Vector SOCPs}
\noindent \textbf{Algorithms Setup}:
For all of the considered algorithms, 
we set the number of batches to $100$,
and the learning rate to $\lambda \in [10^{-2}, 3\times 10^{-2}]$.
For the Ad-GPro, we use a neural network parameterization
to estimate the adjoint\,/\,control processes $(\textbf{y}_t,\{\textbf{z}_t^i\}_{i=1}^l)$.
For this, we apply 3 hidden layers with 16 neurons and \emph{tanh} activation function.

\noindent\textbf{Experiment 1:}
We consider the following SOCP with 3-dimensional controlled process \cite{archibald2020stochastic}
\begingroup\makeatletter\def\f@size{8.5}\check@mathfonts
\begin{align}\label{Eq:vector-exp1}
	&\min_{u\in\mathcal{U}} J(u,\textbf{x}_0) = \frac12 \left[ \int_0^T \mathbb{E}\Big\{ (\textbf{x}_t - \textbf{x}_t^*)^\top \textbf{C} (\textbf{x}_t - \textbf{x}_t^*) \Big\} dt + \int_0^T u_t^2 dt \right], \nonumber \\
	&\text{s.t.} \quad 	d\textbf{x}_t = (u_t \textbf{1} - \mathbf{r}_t)dt + \textbf{I} d\textbf{w}_t, \qquad \textbf{x}_t \in \mathbb{R}^3,
\end{align}
\endgroup
with $\textbf{x}_0=-\textbf{1}$,
where $\textbf{1} = (1,1,1)^\top$, $u_t \in \mathbb{R}$ is the control process,  
$\mathbf{r}_t = \frac{t}{2} \textbf{1}$,
$\mathbf{C} = \begin{pmatrix}
	3 & 0 & 0 \\ 0 & 1 & 0 \\ 0 & 0 & 2
\end{pmatrix}$
and $ \textbf{x}_t^* =  (3Tt - \frac{t^2}{2}) \textbf{1} + \left(-\frac12, 0, 1 \right)^\top $.
It has the following analytical solution for the control decision.
$$
u_t^a = 3T - t/2 - \frac{ 2.5 \cosh(\sqrt{6} t) + \sqrt{6} \sinh(\sqrt{6} (t - T)) } { \cosh(\sqrt{6} T) }
$$

\begin{figure}[t!]
	\centering
	\subcaptionbox{\label{Fig:vector-exp1-objective}}
	{
		\includegraphics[width=0.35\textwidth]{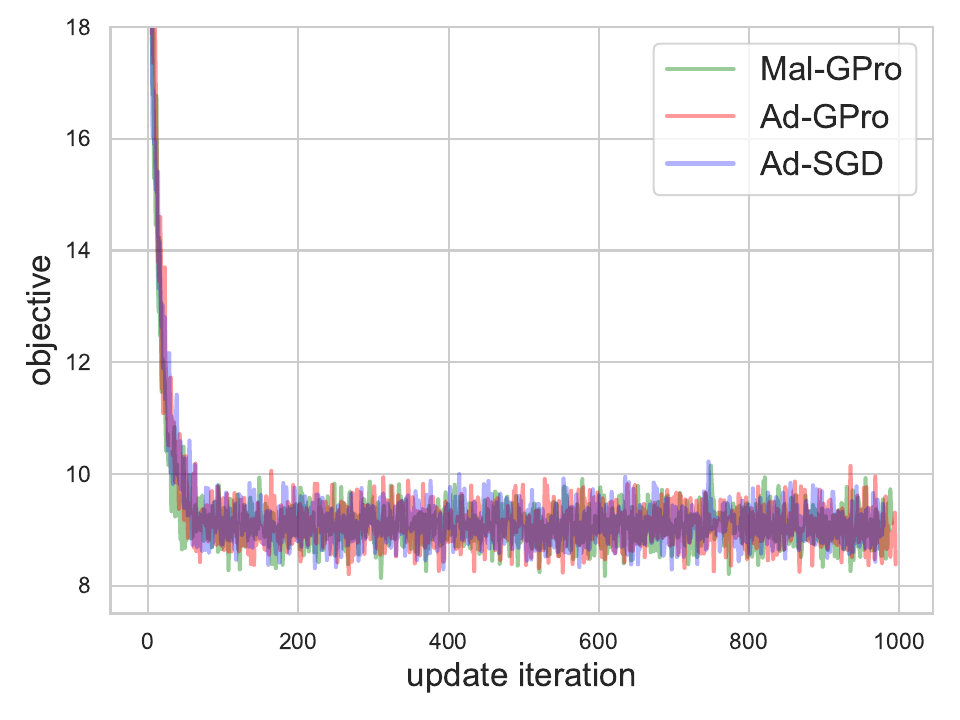}
	}
	\hspace{12 pt}
	\subcaptionbox{\label{Fig:vector-exp1-control}}
	{
		\includegraphics[width=0.35\textwidth]{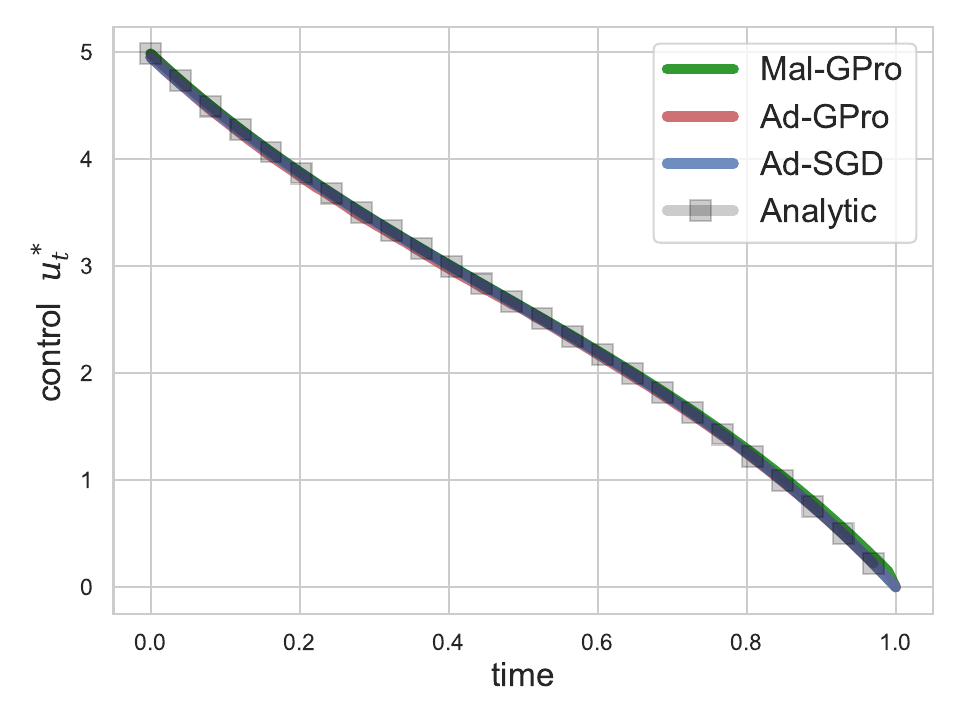}
	}
	\subcaptionbox{\label{Fig:vector-exp1-error}}
	{
		\begin{tabular}{l l}
			\toprule
			Model & Control Error $\mathcal{E}_c$\\
			\cmidrule(lr){1-1}\cmidrule(lr){2-2}
			Ad-GPro,  \scalebox{0.8}{($\Delta t = 10^{-2}$)}	 			 &  $8.2 \times 10^{-4} \pm 1.3 \times 10^{-4}$ \\
			Ad-SGD,   \scalebox{0.8}{($\Delta t = 10^{-2}$)} 		 		 &  $2.2 \times 10^{-4} \pm 1.5\times10^{-4}$ \vspace{3pt}\\ 
			Mal-GPro, \scalebox{0.8}{($\Delta t = 10^{-2}$)} 				 &  $6.5 \times 10^{-4} \pm 9\times10^{-5} $  \\ 
			Mal-GPro, \scalebox{0.8}{($\Delta t = 5\!\times\!10^{-3}$)}  	 &  $2.5 \times 10^{-4} \pm 1 \times 10^{-4}$ \\
			\bottomrule
		\end{tabular}
	}
	\caption{The results of our Mal-GPro algorithm compared to the baseline algorithms, Ad-SGD and Ad-GPro, 
		for problem (\ref{Eq:vector-exp1}).	
		(a) shows the optimum value of objective and 
		(b) the optimum control decision. 
		Clearly, all of the considered algorithm can properly learn the control decision
		and their respective solutions match with the analytical one.
		(c) tabulates the control error $\mathcal{E}_c$
		of different schemes; Mal-GPro algorithm needs smaller step-size to reach a level of error comparable to  Ad-SGD and Ad-GPro.
	}
	\label{Fig:vector-Exp1}
\end{figure}

\cref{Fig:vector-Exp1} shows the results of our Mal-GPro algorithm compared to the baseline algorithms, Ad-SGD and Ad-GPro.
Clearly, all of the considered algorithm 
have the same behavior from performance and sample-efficiency viewpoints, 
and can properly learn control decisions matching with the analytical solution.
\cref{Fig:vector-exp1-error} tabulates the \emph{control error}
of different schemes; 
Mal-GPro algorithm needs smaller step-size to reach a level of error comparable to  Ad-SGD.

\noindent\textbf{Experiment 2:}
We now consider the following vector problem \cite{Mohamed2022}:
\begingroup\makeatletter\def\f@size{8.5}\check@mathfonts
\begin{align}
	\label{Eq:vector-exp2}
	&\min_{\textbf{u}\in\mathcal{U}}  J(\textbf{u},\textbf{x}_0) := \mathbb{E}
	\left\{ 
	\int_0^T (\textbf{x}_t^\top \textbf{R}\textbf{x}_t + \textbf{u}_t^\top \textbf{C} \textbf{u}_t) dt 
	\right\} \notag \\
	&~~\text{s.t.}~~ d\textbf{x}_t =\textbf{a}(\textbf{x}_t, \textbf{u}_t) dt + \textbf{b}(\textbf{x}_t,\textbf{u}_t) dw_t,
\end{align}
\endgroup
where 
$\textbf{x}_0 = -\textbf{1}$, $T=1$,
$\textbf{R}=\begin{pmatrix}
	1 & 0 \\ 0 & 2
\end{pmatrix}$, 
$\textbf{C}=\begin{pmatrix}
	1 & 0 \\ 0 & 4
\end{pmatrix}$, 
$\textbf{a}(\textbf{x}_t,\textbf{u}_t)=\begin{pmatrix}
	-x_{t,0} -2x_{t,1}^2 - x_{t,2}^3/2 + u_{t,0} \\ -\cos(x_{t,1}) + u_{t,1}
\end{pmatrix}$,
$\textbf{b}(\textbf{x}_t,\textbf{u}_t)=\begin{pmatrix}
	0.4 + u_{t,0} \\ 0.2 + 2u_{t,1}
\end{pmatrix}$ 
and $\mathcal{U} = \{\textbf{u} \in U \subseteq \mathbb{R}^2 \:|\:-1 \leq u_i \leq 1,\: i\in\{0,1\} \}$.

\begin{figure}[t!]
	\centering
	\hspace{-12 pt}
	\subcaptionbox{\label{Fig:vector-exp2-action-MalGPro}}
	{
		\includegraphics[width=0.33\textwidth]{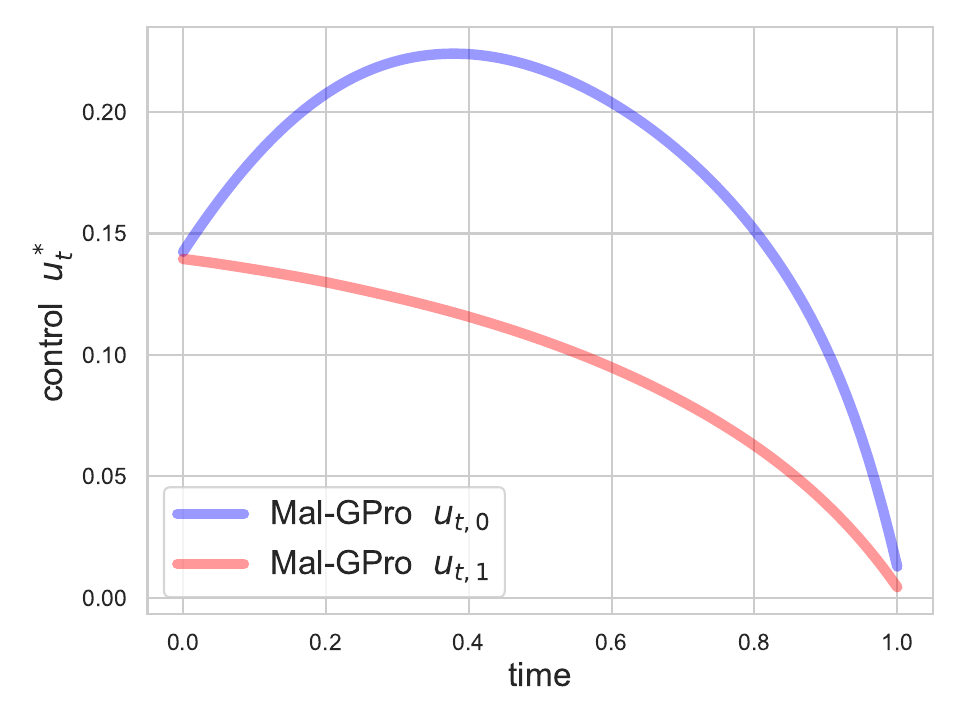}
	}
	\hspace{-8 pt}
	\subcaptionbox{\label{Fig:vector-exp2-action-AdSGD}}
	{
		\includegraphics[width=0.33\textwidth]{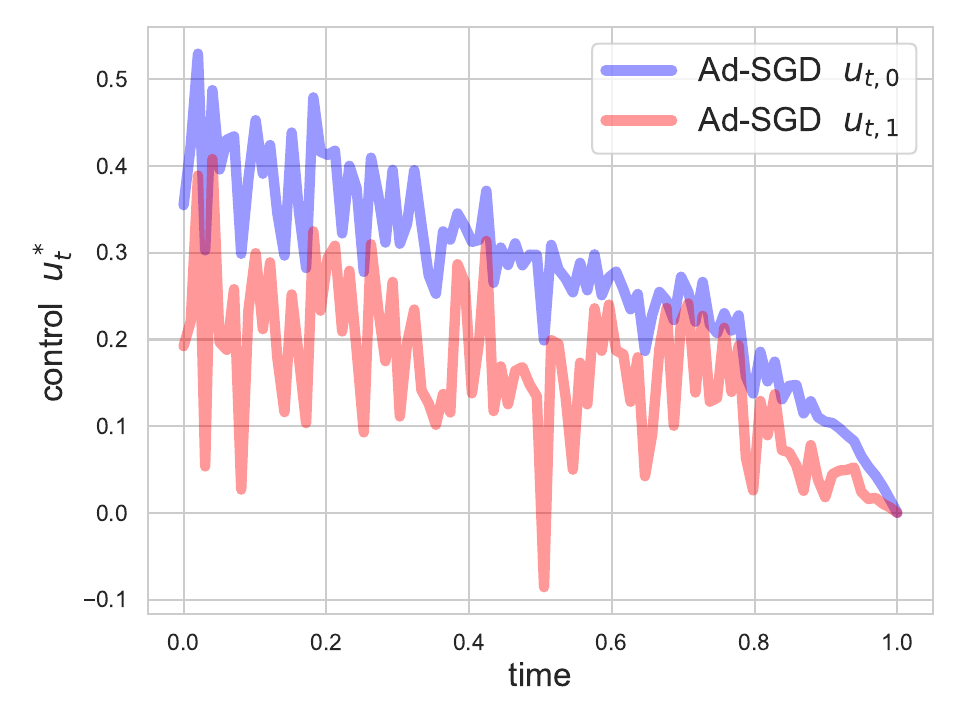}
	}
	\hspace{-8 pt}
	\subcaptionbox{\label{Fig:vector-exp2-action-AdGPro}}
	{
		\includegraphics[width=0.33\textwidth]{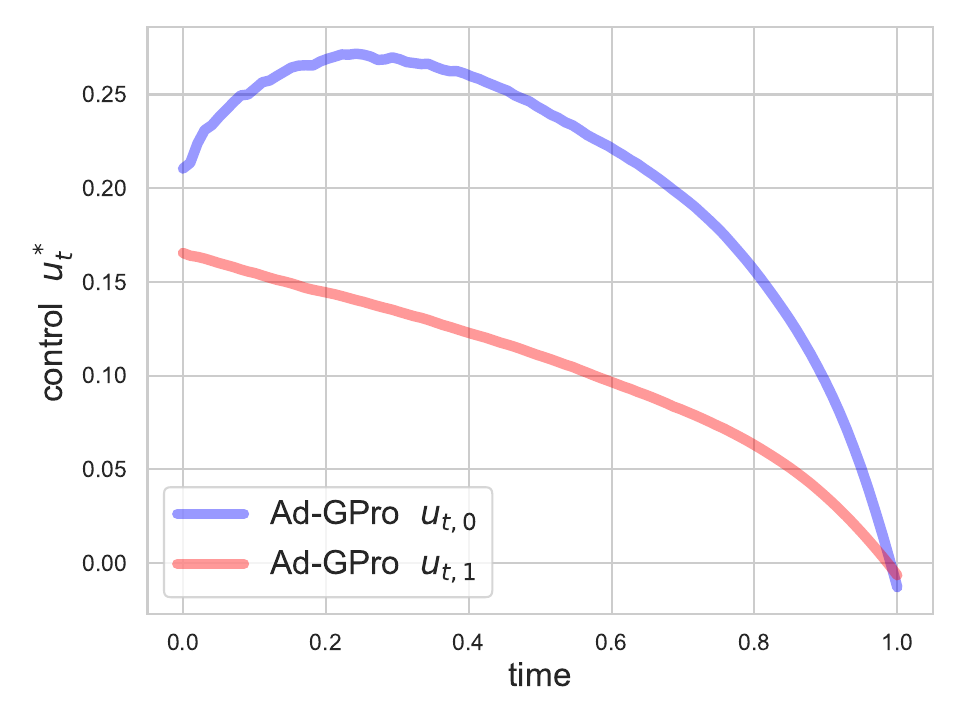}
	}
	\label{Fig:vector-exp2}
	\caption{The performance of Mal-GPro  algorithms on SOCP~(\ref{Eq:vector-exp2}), 
		compared to the baseline approaches Ad-GPro and Ad-SGD.
		(a) shows the control vector by Mal-GPro, 
		(b) the control vector found by Ad-SGD, and
		(c) the control vector by Ad-GPro.
		In contrast to Mal-GPro, Ad-SGD cannot properly find the optimum control decision.}
\end{figure}
\cref{Fig:vector-exp2-action-MalGPro,Fig:vector-exp2-action-AdSGD,Fig:vector-exp2-action-AdGPro} illustrate
the obtained control decision for Ad-SGD, Mal-GPro and Ad-GPro, respectively, on SOCP~(\ref{Eq:vector-exp2}).
In contrast to Mal-GPro, Ad-SGD cannot properly find the optimum control decision.
There is slight difference between control solutions of Mal-GPro and Ad-GPro
due to the limited capacity of the parametric methodology in Ad-GPro.
It again shows the effectiveness of Mal-GPro in obtaining optimal solution for SOCPs,
compared to the baseline algorithms.

\noindent\textbf{Experiment 3:}
We now consider a Linear-Quadratic (LQ) control problem as a high-dimensional SOCP benchmark \citep{Pham2009,SunYong2020}:
\begin{align}
	\label{Eq:vector-exp3}
	&\min_{\textbf{u}\in\mathcal{U}}  J(\textbf{u},\textbf{x}_0) := \mathbb{E}
	\left\{ 
	\frac12 \int_0^T (\textbf{x}_t^\top \textbf{Q}\textbf{x}_t + \textbf{u}_t^\top \textbf{R} \textbf{u}_t) dt 
	\right\} \notag \\
	&~~\text{s.t.}~~ d\textbf{x}_t = \big(\textbf{A}\:\textbf{x}_t  + \textbf{B} \textbf{u}_t \big)dt 
	+ \boldsymbol{\Sigma}\:  d\textbf{w}_t,
\end{align}
For this experiment, we set $n = k = 10$, 
$\textbf{A} = -0.5 \textbf{U}_1$ and $\textbf{B} = \textbf{U}_2$, 
where $\{\textbf{U}_i\}_{i=1}^2$ are matrices with elements uniformly selected within $[0,1]$. 
Moreover, $\textbf{Q} = \textbf{I}$, 
$\textbf{R} = 0.1 \textbf{I}$,
and $\boldsymbol{\Sigma} = \nu  \textbf{I}$
where $\nu= 0.3$ is the volatility coefficient \citep{ControlMatching2024}.
Note that this problem, known as linear quadratic regulator (LQR), have an analytical control solution
that can be obtained based on the Riccati equation \citep{SunYong2020}.

\begin{figure}[t!]
	\centering
	\subcaptionbox{\label{Fig:vector-exp3-Analyticalcontrol}}
	{
		\includegraphics[width=0.3\textwidth]{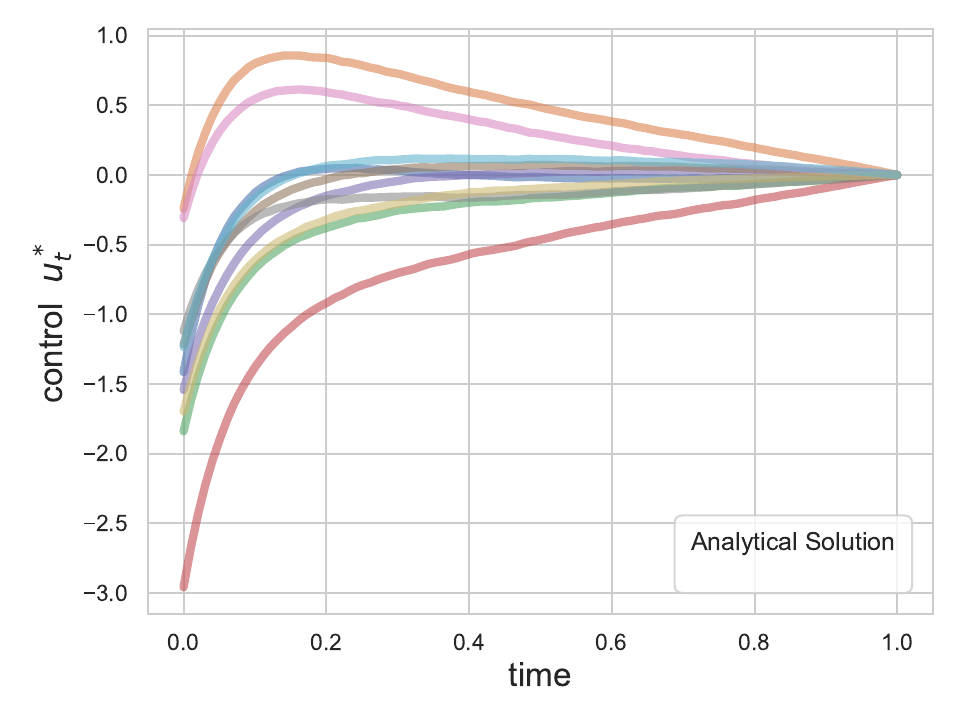}
	}
	\subcaptionbox{\label{Fig:vector-exp3-MallGcontrol}}
	{
		\includegraphics[width=0.3\textwidth]{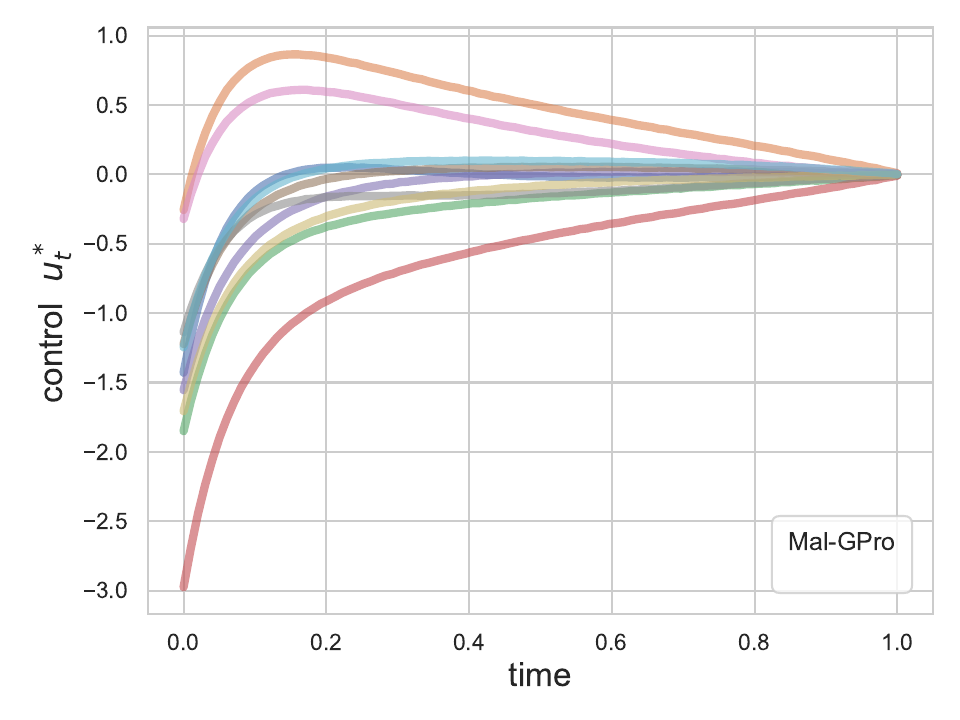}
	}
	\subcaptionbox{\label{Fig:vector-exp3-SGDcontrol}}
	{
		\includegraphics[width=0.3\textwidth]{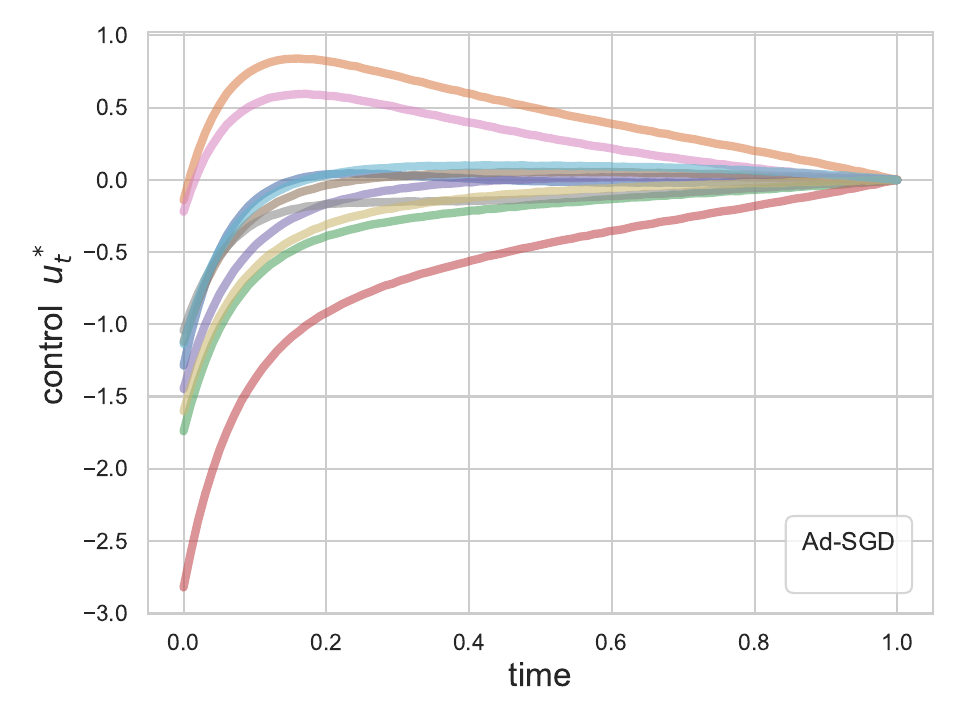}
	}
	\subcaptionbox{\label{Fig:vector-exp3-error}}
	{
		\small
		\begin{tabular}{l l}
			\toprule
			Model & Control Error $\mathcal{E}_c$\\
			\cmidrule(lr){1-1}\cmidrule(lr){2-2}
			Ad-SGD    		 		 &  $3.1 \times 10^{-3} \pm 1.0\times10^{-5}$ \vspace{3pt}\\ 
			Mal-GPro 				 &  $8.7 \times 10^{-4} \pm 4.6\times10^{-6} $  \\ 
			\bottomrule
		\end{tabular}
		\small
	}
	\caption{The results of our Mal-GPro algorithm compared to the baseline algorithm Ad-SGD 
		for problem (\ref{Eq:vector-exp3}) with $\Delta t = 10^{-2}$.	
		(a) shows the analytical control solution,
		and 
		(b)-(c) show the  control decisions
		obtained   by Mal-GPro and Ad-SGD. 
		Clearly, all of them find the same control solutions.
		(d) tabulates the control error $\mathcal{E}_c$
		of Mal-GPro against Ad-SGD.
		Our algorithm develops lower control error with less variance compared to Ad-SGD for this problem.
	}
	\label{Fig:vector-exp3}
\end{figure}
Although, we do not evaluate Ad-GPro  due to its computational complexity,
the analytical solution can be obtained as a benchmark for this problem.
\cref{Fig:vector-exp3} illustrates the analytical control decision and learned decisions of Mal-GPro algorithm and baseline Ad-SGD on SOCP~(\ref{Eq:vector-exp3}).
Clearly, all of them find the same control decision.
However, \cref{Fig:vector-exp3-error} tabulates the control error $\mathcal{E}_c$
of Mal-GPro against Ad-SGD.
Our algorithm develops lower error and with less variance compared to the Ad-SGD for this problem.

\begin{table}[t]
	\centering
	\caption[Average per-iteration runtime]{Average runtime per iteration update (seconds) 
		for our method Mal-GPro versus baselines Ad-GPro and Ad-SGD evaluated on SOCPs (\ref{Eq:vector-exp1}),(\ref{Eq:vector-exp2})
		and (\ref{Eq:vector-exp3}) with $n=k=5$.}
	\label{Tab:runtime}
	\small
	\begin{tabular}{@{} l S[table-format=2.2] S[table-format=2.2] S[table-format=2.2] @{}}
		\toprule
		Problem & {Mal-GPro} & {Ad-SGD} & {Ad-GPro} \\
		\midrule
		SOCP~(\ref{Eq:vector-exp1}) & 0.5 & 0.1 & 12.0 \\
		SOCP~(\ref{Eq:vector-exp2}) & 0.3 & 0.1 & 11.0 \\
		SOCP~(\ref{Eq:vector-exp3}) & 0.8 & 0.1 & 26.6 \\
		\bottomrule
	\end{tabular}
	
	\vspace{1ex}
	\footnotesize
	\textbf{Note:} Reported values are per-iteration averages computed over independent runs (mean over iterations). 
	Measurements were taken on CPU \texttt{Core i5-13600H} with 32 GB RAM.
\end{table}
\noindent \textbf{Average runtime:}
We now benchmark different algorithms from average runtime perspective, i.e., the average time required for methods run one iteration update.
For this, we consider experiments in \cref{Eq:vector-exp1,Eq:vector-exp2,Eq:vector-exp3}.
\cref{Tab:runtime} compares our algorithm Mal-GPro against baseline methods Ad-GPro and Ad-SGD in terms of average runtime metric.
Mal-GPro and Ad-SGD are time-efficient
while Ad-GPro is considerably more time-consuming.
Ad-SGD has been able to perform with least runtime
by sacrificing the precision as the results in \cref{Fig:scalar-exp2-control,Fig:vector-exp2-action-AdSGD} illustrate.

These results altogether show the effectiveness of Mal-GPro compared to baseline methods.
Specifically, it highlights its efficiency against Ad-SGD from the performance viewpoint
as well as against Ad-GPro from the computational complexity perspective.

\section{Conclusion}
In this paper, we leveraged Malliavin calculus to derive a stochastic maximum principle
for stochastic optimal control problems (SOCPs).
Accordingly, we derived a gradient projection algorithm to sequentially solve some classes of SOCPs.
The experimental results show that our algorithm,
at the same time can provide competitive performance against the standard iterative algorithms,
and does not suffer from the dimensionality challenge of the standard approaches.
The reason is that our algorithm replaces the need of backward SDEs with a Malliavin derivative
that hopefully can be computed accurately based on a forward simulator. 

\bibliography{malliavin.bib}
\bibliographystyle{icml2024}

\newpage
\appendix
\onecolumn

\section{Proof of \cref{Thm1}}
\label{App1}
To prove \cref{Thm1}, we need the following Lemma:
\begin{lemma}\label{Lemma1}
	Consider the diffusion process $X=(x_t)_{t\in[0,T]}$ following the SDE~(\ref{Eq:SDE}) with $n=k=1$, i.e.,
	$$
	dx_t = a(x_t,u_t)dt+ b(x_t,u_t) dw_t.
	$$
	The solution of the variational SDE
	\begin{align}
	\label{Eq:Aux1}
	d\delta_t = \big( a_x(x_t,u_t)\delta_t +  a_u(x_t,u_t)v_t \big)dt + \big( b_x(x_t,u_t)\delta_t +  b_u(x_t,u_t)v_t \big)dw_t.
	\end{align}
	with $\delta_0=0$ is then expressed by
	$$
	\delta_t = \int_0^t \frac{D_sx_t}{b(x_s,u_s)} \big( a_u(x_s,u_s)-b_x(x_s,u_s)\:b_u(x_s,u_s) \big)\:v_s\:ds 
				+\int_0^t \frac{D_sx_t}{b(x_s,u_s)} b_u(x_s,u_s)\:v_s\:dw_s .
	$$
	\begin{proof}
		We first consider a SDE with the form of
		\begin{align}
		\label{Eq:Aux2}
		d\eta_t =  a_x(x_t,u_t)\eta_t  \: dt +  b_x(x_t,u_t)\eta_t \: dw_t.
		\end{align}
		By considering the variable transformation $e_t = \log(\eta_t)$ and applying Ito's formula, 
		we can simply obtain the solution of $\eta_t$ as
		\begin{align}
		\label{Eq:eta}
			\eta_t = \exp\bigg( \int_0^t \big( a_x(x_t,u_t)-\frac12 b_x(x_t,u_t )^2 \big) dt + \int_0^t b_x(x_t,u_t) dw_t \bigg) > 0.
		\end{align}
		Based on the variation-of-constant approach, the solution of $\delta_t$ can be obtained using 
		\begin{align}\label{Eq:constant-of-variation}
			\delta_t = \eta_t q_t,
		\end{align}
		where $q_t$ should be found.
		For $b_u = 0$, Ito's formula and some mathematical manipulations
		show that $dq_t = \eta_t^{-1}(a_u(x_t,u_t)v_t\:dt + b_u(x_t,u_t)v_t \:dw_t)$.
		However, for $b_u \ne 0$, we should set: 
		\begin{align}
		\label{Eq:Aux3}
		\eta_t\:dq_t = (a_u(x_t,u_t)v_t + \gamma_t)\:dt + (b_u(x_t,u_t)v_t+\mu_t) \:dw_t,
		\end{align}
		where $\gamma_t$ and $\mu_t$ should be determined so that the dt-terms match after including the quadratic covariation.
		To obtain them, \cref{Eq:constant-of-variation} gives:
		$$
		d\delta_t = d\eta_t\:q_t + \eta_t\:dq_t + d\eta_t\:dq_t,
		$$
		which accordingly gives the following from \cref{Eq:Aux1,Eq:Aux2,Eq:Aux3}:
		\begin{align*}
		&\big( {\color{darkred}\cancel{a_x(x_t,u_t)\delta_t}} +  {\color{blue}\cancel{a_u(x_t,u_t)v_t}} \big)dt 
		+ \big( {\color{purple}\cancel{b_x(x_t,u_t)\delta_t}} +  {\color{darkgreen}\cancel{b_u(x_t,u_t)v_t}} \big)dw_t \\
		&\qquad=  (a_x(x_t,u_t)\eta_t  \: dt +  b_x(x_t,u_t)\eta_t \: dw_t)q_t
		 + (a_u(x_t,u_t)v_t + \gamma_t)\:dt + (b_u(x_t,u_t)v_t+\mu_t) \:dw_t\\
		 &\qquad~~+ b_x(x_t,u_t)\eta_tdw_t\:\eta_t^{-1}(b_u(x_t,u_t)v_t + \mu_t)dw_t \\
		&\qquad=  ( {\color{darkred}\cancel{a_x(x_t,u_t)\delta_t}}  \: dt +  {\color{purple}\cancel{b_x(x_t,u_t)\delta_t}} \: dw_t)
		+ ( {\color{blue}\cancel{a_u(x_t,u_t)v_t}} + \gamma_t)\:dt + ({\color{darkgreen}\cancel{b_u(x_t,u_t)v_t}}+\mu_t) \:dw_t \\
		&\qquad~~+ b_x(x_t,u_t)(b_u(x_t,u_t)v_t + \mu_t)dt.
		\end{align*}
		By canceling the similar terms, we then obtain:
		$$
		\gamma_t = -b_x(x_t,u_t) b_u(x_t,u_t) v_t,\qquad\mu_t=0.
		$$
		Considering that $q_0 = \eta_0^{-1}\delta_0 = 0$, we finally get:
		\begin{align}
		\label{Eq:delta}
		\delta_t = \eta_t q_t = \int_0^t \eta_t\eta_s^{-1} \big( a_u(x_s,u_s)-b_x(x_s,u_s)\:b_u(x_s,u_s) \big)\:v_s\:ds 
		+\int_0^t \eta_t\eta_s^{-1} b_u(x_s,u_s)\:v_s\:dw_s .
		\end{align}
		Now, based on \cref{Eq:eta}, one can get:
		$$
		\eta_t\eta_s^{-1}  = \exp\bigg( \int_s^t \big( a_x(x_t,u_t)-\frac12 b_x(x_t,u_t )^2 \big) dt + \int_s^t b_x(x_t,u_t) dw_t \bigg),
		$$
		which gives $\eta_t\eta_s^{-1} = \frac{D_sx_t}{b(x_s,u_s)}$  based on \cref{Eq:D_sx_t}.
		By substituting this into \cref{Eq:delta}, the statement thus follows.
	\end{proof}
\end{lemma}

We now present \cref{Thm1} with a proof.
Note that without loss of generalization, we assume that $h(x_T) = 0$.
\begin{theorem}\label{Thm1_app}
	Consider SOCP $\mbox{P}_1$ (\ref{EQ:StochasticControl}) with $n=k=1$,
	the variation of cost functional $J(u,x_0)$
	along direction $v$, through Definition~(\ref{Def1}), can be obtained by:
	\begingroup\makeatletter\def\f@size{9}\check@mathfonts
	\begin{align*}
	\delta_v J(u,x_0) 
	= \int_0^T \mathbb{E}
	\Big\{ &
	\int_s^T \Big( \frac{a_u(s) - b_x(s)b_u(s)}{b(s)} L_x(t) D_sx_t \\
	&~~+  \frac{b_u(s)}{b(s)} \big( L_{xx}(t) D_sx_t + L_{xu}(t) D_su_t \big) D_sx_t\Big) dt +L_u(s) 
	\Big\} v_s ds,
	\end{align*}
	\endgroup
	where $D_sx_t$ is the Malliavin derivative of $x_t$ against $w_s$,
	$L_x(t)$, $a_u(s)$, and $b(s)$ are short notations 
	for $L_x(x_t,u_t,t)$, $a_u(x_s,u_s)$ and $b(x_s,u_s)$, respectively,
	with $L_x$ and $L_{xx}$ denoting the first and second gradients of $L$ 
	w.r.t $x$, and
	$L_u$, $a_u$, $b_u$ showing the gradient of $L$, $a$ and $b$ w.r.t. $u$.
	\begin{proof}
		We perturb the control process $u_t$ through $u_t+  \epsilon v_t$, 
		which leads to the diffusion process $x_t^{u_t+  \epsilon v_t}$.
		We then define a variational process:
		$$
		\delta_t^x := \lim_{\epsilon \to 0}\frac{1}{\epsilon}(x_t^{u+  \epsilon v} - x_t^{u}).
		$$
		Based on \cref{Eq:SDE}, $\delta_t^x $ obeys the following variational SDE:
		\begin{align}\label{Eq:variationa_x}
		d\delta_t^x = \big( a_x(x_t,u_t)\delta_t^x +  a_u(x_t,u_t)v_t \big)dt + \big( b_x(x_t,u_t)\delta_t^x +  b_u(x_t,u_t)v_t \big)dw_t.
		\end{align}
		Using Lemma (\ref{Lemma1}), the solution of this variational SDE is:
		\begin{align}\label{Eq:solution-delta_x}
		\delta_t^x = \int_0^t \frac{D_sx_t}{b(x_s,u_s)} (a_u(x_s,u_s)-b_x(x_s,u_s)\:b_u(x_s,u_s))\:v_s\:ds 
		            +\int_0^t \frac{D_sx_t}{b(x_s,u_s)} b_u(x_s,u_s)\:v_s\:dw_s .
		\end{align}
		On the other hand, the variation of cost functional $J(u,x_0)$ along direction $v$ obeys:
		\begin{align}\label{Eq:variation-of-cost}
		\delta_vJ(u,x_0) = \mathbb{E}\Big\{ \int_0^T  \big(L_x(x_t,u_t,t) \delta_t^x + L_u(x_t,u_t,t) v_t\big) dt\Big\} .
		\end{align}
		By plugging \cref{Eq:solution-delta_x} into \cref{Eq:variation-of-cost}, we then have:
		\begin{align*}
		\delta_vJ(u,x_0) =& \mathbb{E} \int_0^T L_x(x_t,u_t,t) \int_0^t \frac{D_sx_t}{b(x_s,u_s)} \big( a_u(x_s,u_s) - b_x(x_s,u_s)b_u(x_s,u_s) \big) v_s ds \: dt \\
					      &+\mathbb{E} \int_0^T L_x(x_t,u_t,t) \int_0^t \frac{D_sx_t}{b(x_s,u_s)} b_u(x_s,u_s) v_s dw_s \: dt			      
					       +\mathbb{E} \int_0^T L_u(x_t,u_t,t)  v_t dt.
		\end{align*}
		We then leverage the integration-by-parts~(\ref{Eq:integration-by-parts}) for the second term in RHS of recent equation:
		\begin{align*}
			\delta_vJ(u,x_0) =& \mathbb{E} \int_0^T \int_0^t L_x(x_t,u_t,t)\frac{D_sx_t}{b(x_s,u_s)} (a_u(x_s,u_s) - b_x(x_s,u_s)b_u(x_s,u_s))\: v_s \: ds  \: dt \\
			&+\mathbb{E} \int_0^T \int_0^t D_sL_{x}(x_t,u_t,t) \frac{D_sx_t}{b(x_s,u_s)} b_u(x_s,u_s)\: v_s \: ds  \: dt			      
			+\mathbb{E} \int_0^T L_u(x_t,u_t,t)  v_t dt.
		\end{align*}		
		Changing the order of integrals then gives:
		\begin{align*}
		\delta_vJ(u,x_0) =& \mathbb{E} \int_0^T \int_s^T L_x(x_t,u_t,t)\frac{D_sx_t}{b(x_s,u_s)} (a_u(x_s,u_s) - b_x(x_s,u_s)b_u(x_s,u_s)) \: dt \: v_s  \: ds \\
		&+\mathbb{E} \int_0^T \int_s^T D_sL_{x}(x_t,u_t,t) \frac{D_sx_t}{b(x_s,u_s)} b_u(x_s,u_s) \: dt \: v_s  \: ds
		+\mathbb{E} \int_0^T L_u(x_t,u_t,t)  v_t dt.
		\end{align*}
		where we used the chain rule $D_s L_x(x_t,u_t,t) = L_{xx}(x_t,u_t,t) D_sx_t +  L_{xu}(x_t,u_t,t) D_su_t$.
		Therefore, the statement proves.
	\end{proof}
\end{theorem}

\section{Proof of \cref{Thm2}}
\label{App2}

To prove \cref{Thm2}, we need some founding Lemmas.
Before that we need to first recall the Stein's lemma that for zero-mean jointly Gaussian random variables 
$\textbf{x}\in\mathbb{R}^n$ and $\textbf{y}\in\mathbb{R}^n$,
and a vector-valued continuously differentiable function $F:\mathbb{R}^n \to \mathbb{R}^n$, we can get:
\begin{align}\label{Eq:SteinLemma}
	\mathbb{E}\left\{ F(\textbf{x})^\top \textbf{y} \right\} 
	= {\rm tr}\left\{ \mathbb{E}\left\{ \mathcal{J}_\textbf{x}F(\textbf{x}) \right\} \mathbb{E}\left\{\textbf{x}\textbf{y}^\top \right\} \right\}
\end{align}

\begin{lemma}
	\label{lemma:integation-by-parts}
	Suppose $F:\mathbb{R}^n \to \mathbb{R}^n$ be a vector-valued random function in the domain of Malliavin derivative,
	$\textbf{X}=(\textbf{x}_t)_{t}$ be a Wiener-driven and Malliavin differentiable process, 
	and $\{ \textbf{G}^l=(\textbf{g}^l_t)_t\}_{l=1}^d$ with $\textbf{g}^l_t \in \mathbb{R}^n$ 
	be an adapted and square-integrable process, we then get:
	$$
	\mathbb{E}\left\{ F(\textbf{x}_t)^\top \sum_{l=1}^d \int_0^t \textbf{g}^l_s dw_s^l \right\} 
	= \sum_{l=1}^d \int_0^t \mathbb{E}\left\{ D_s^l F(\textbf{x})^\top \:  \textbf{g}^l_s \right\} ds
	= \sum_{l=1}^d \int_0^t \mathbb{E}\left\{ {D_s^l\textbf{x}_t}^\top \mathcal{J}_\textbf{x}F(\textbf{x})^\top\:  \textbf{g}^l_s \right\} ds,
	$$
	where $D_s^l\textbf{x}_t$ is the Malliavin derivative of $\textbf{x}_t$ against the Wiener process $w_t^l$.
	\begin{proof}
		This theorem is the multivariate extension of integration-by-parts formula (see \cref{Sec:MalliavinDerivative}) 
		in the Malliavin calculus.
		Without loss of generality, we provide a proof for the case of
		$\textbf{x}_t = \int_0^t \sum_{l=1}^d \textbf{h}_s^l dw_s^l$.
		The proof can be then generalized for the other cases \citep{nualart2018introduction,alos2021malliavin}.
		Let $\mathcal{F}^{\textbf{h},\textbf{g}} $ be the $\sigma$-algebra of processes $\textbf{h}^l$ and $\textbf{g}^l$
		that encodes all the information of these process up to time $t$.
		We then get:
		\begingroup\makeatletter\def\f@size{9.5}\check@mathfonts
		\begin{align*}
			&\mathbb{E} \left\{ F(\textbf{x}_t)^\top \sum_{l=1}^d \int_0^t \textbf{g}^l_s dw_s^l \right\} 		
			=\mathbb{E}_{\mathcal{F}^{\textbf{h},\textbf{g}} } \mathbb{E}\left\{ F(\textbf{x}_t)^\top \sum_{l=1}^d \int_0^t \textbf{g}_s^l dw_s^l  ~\Big|~ \mathcal{F}^{\textbf{h},\textbf{g}} \right\}\\
			&\hspace{60pt} \overset{(a)}{=} \mathbb{E}_{\mathcal{F}^{\textbf{h},\textbf{g}}}\:{\rm tr}\left\{  
			\sum_{l=1}^d \mathbb{E}\left\{ \mathcal{J}_\textbf{x}F \:\big|\: \mathcal{F}^{\textbf{h},\textbf{g}} \right\} 
			\mathbb{E}\left\{ \int_0^t \textbf{h}_s^l {\textbf{g}_s^l}^\top ds  \:\Big|\: \mathcal{F}^{\textbf{h},\textbf{g}} \right\}\right\}\\
			&\hspace{60pt} = \mathbb{E}_{\mathcal{F}^{\textbf{h},\textbf{g}}}\:{\rm tr}\left\{  
			\sum_{l=1}^d \mathbb{E}\left\{ \mathcal{J}_\textbf{x}F \:\big|\: \mathcal{F}^{\textbf{h},\textbf{g}} \right\} 
			\int_0^t \textbf{h}_s^l {\textbf{g}_s^l}^\top ds  \right\}\\
			&\hspace{60pt} = \mathbb{E}_{\mathcal{F}^{\textbf{h},\textbf{g}}}\:{\rm tr}\left\{  
			\sum_{l=1}^d \mathbb{E}\left\{ \mathcal{J}_\textbf{x}F \: \int_0^t \textbf{h}_s^l {\textbf{g}_s^l}^\top ds \:\Big|\: \mathcal{F}^{\textbf{h},\textbf{g}} \right\} 	\right\}\\
			&\hspace{60pt} \overset{(b)}{=}  \mathbb{E}_{\mathcal{F}^{\textbf{h},\textbf{g}}}\:\sum_{l=1}^d \int_0^t \mathbb{E}\left\{ {\textbf{g}_s^l}^\top\mathcal{J}_{\textbf{x}}F\:D_s^l\textbf{x}_t \:\Big|\: \mathcal{F}^{\textbf{h},\textbf{g}}\right\} ds
			= \sum_{l=1}^d \int_0^t \mathbb{E}\left\{ {\textbf{g}_s^l}^\top\mathcal{J}_{\textbf{x}}F\: D_s^l\textbf{x}_t \right\} ds
		\end{align*}
		\endgroup
		where (a) obtained based on the Stein's lemma~(\ref{Eq:SteinLemma}), and 
		we used ${\rm tr}\{\textbf{A} \textbf{z}\textbf{z}^\top \} = {\rm tr}\{\textbf{z}^\top \textbf{A} \textbf{z}\}$ 
		for (b) with vector $\textbf{z}$ and matrix $\textbf{A}$.
	\end{proof}
\end{lemma}

\begin{lemma}\label{Lemma2}
	Consider the diffusion process $\textbf{X}=(\textbf{x}_t)_{t\in[0,T]}$ following the SDE~(\ref{Eq:SDE}), i.e.,
	$$
	d\textbf{x}_t = \textbf{a}(\textbf{x}_t,\textbf{u}_t)dt+ \sum_{l=1}^d \textbf{b}_l(\textbf{x}_t,\textbf{u}_t) dw_t^l.
	$$
	The solution of the variational SDE
	\begin{align}
		\label{Eq:Aux4}
		d\boldsymbol{\delta}_t = \big( \mathcal{J}_\textbf{x}\textbf{a}(\textbf{x}_t,\textbf{u}_t)\boldsymbol{\delta}_t 
		+  \mathcal{J}_\textbf{u}\textbf{a}(\textbf{x}_t,\textbf{u}_t)\textbf{v}_t \big)dt + \sum_{l=1}^d\big( \mathcal{J}_\textbf{x}\textbf{b}_l(\textbf{x}_t,\textbf{u}_t)\boldsymbol{\delta}_t +  \mathcal{J}_\textbf{u}\textbf{b}_l(\textbf{x}_t,\textbf{u}_t)\textbf{v}_t \big)dw_t^l.
	\end{align}
	with $\boldsymbol{\delta}_0=\textbf{0}$, is then expressed by
	$$
	\boldsymbol{\delta}_t = \int_0^t \boldsymbol{\Gamma}_{s,t}\Big( \mathcal{J}_\textbf{u}\textbf{a}(\textbf{x}_s,\textbf{u}_s)
	-\sum_{l,l'}^d q_{l,l'}\mathcal{J}_\textbf{x}\textbf{b}_l(\textbf{x}_s,\textbf{u}_s)  \mathcal{J}_\textbf{u}\textbf{b}_{l'}(\textbf{x}_s,\textbf{u}_s) \Big)\:\textbf{v}_s\:ds 
	+\int_0^t \sum_{l=1}^d \boldsymbol{\Gamma}_{s,t} \mathcal{J}_\textbf{u} \textbf{b}_l(\textbf{x}_s,\textbf{u}_s)\:\textbf{v}_s\:dw_s^l,
	$$
	where $q_{l,l'}$ is the diffusion coefficient between $w_t^l$ and $w_t^{l'}$,
	and $\boldsymbol{\Gamma}_{s,t} = \textbf{Y}_t\textbf{Z}_s$ follows the stochastic flows based on \cref{Eq:Y-Z}.
	\begin{proof}
		For the ease of notations, we hereafter denote the arguments of functions, i.e. $(\textbf{x}_t,\textbf{u}_t)$ by merely $(t)$.
		We  then consider a matrix SDE with the form of
		\begin{align}
			\label{Eq:Aux5}
			d\Phi_t =  \mathcal{J}_\textbf{x}\textbf{a}(t) \Phi_t  \: dt +  \sum_{l=1}^d \mathcal{J}_\textbf{x}\textbf{b}_l(t)\Phi_t \: dw_t^l,\quad\Phi_0=\textbf{I},
		\end{align}
		where $\Phi_t \in \mathbb{R}^{n\times n}$ is assumed to be invertible.
		Notice that the solution for $\Phi_t$ that meets all of these conditions
		is the stochastic flow $\textbf{Y}_t$~(\ref{Eq:Y-Z}).
		Based on the variation-of-constant approach, 
		the solution of $\boldsymbol{\delta}_t$ in (\ref{Eq:Aux4}) can be obtained using 
		\begin{align}\label{Eq:constant-of-variation2}
			\boldsymbol{\delta}_t = \Phi_t \textbf{q}_t,
		\end{align}
		where vector $\textbf{q}_t \in \mathbb{R}^n$ should be found.
		For $\mathcal{J}_\textbf{u}\textbf{b}_l = 0$, one can simply
		show that $\Phi_t\:d\textbf{q}_t = \mathcal{J}_\textbf{u}\textbf{a}(t)\textbf{v}_t\:dt 
		+ \sum_{l=1}^d \mathcal{J}_\textbf{u}\textbf{b}_l(t)\textbf{v}_t \:dw_t^l$.
		However, in the case of $\mathcal{J}_\textbf{u}\textbf{b}_l \ne 0$, 
		we should set the following based on the Lemma~(\ref{Lemma1}): 
		\begin{align}
			\label{Eq:Aux6}
			\Phi_t\:d\textbf{q}_t = (\mathcal{J}_\textbf{u}\textbf{a}(t)\textbf{v}_t+ \Psi_t)\:dt 
			+ \sum_{l=1}^d \mathcal{J}_\textbf{u}\textbf{b}_l(t)\textbf{v}_t \:dw_t^l,
		\end{align}
		where $\Psi_t \in \mathbb{R}^{n \times n}$ now should be determined 
		so that the dt-terms match after including the quadratic covariation.
		To obtain it, we take differential of \cref{Eq:constant-of-variation2}:
		$$
		d\boldsymbol{\delta}_t = d\Phi_t\:\textbf{q}_t + \Phi_t\:d\textbf{q}_t + d\Phi_t\:d\textbf{q}_t,
		$$
		which accordingly gives the following based on the \cref{Eq:Aux4,Eq:Aux5,Eq:Aux6}:
		\begin{align*}
			&\big( {\color{purple}\cancel{\mathcal{J}_\textbf{x}\textbf{a}(t)\:\boldsymbol{\delta}_t}} 
			+      {\color{darkred}\cancel{\mathcal{J}_\textbf{u}\textbf{a}(t)\:\textbf{v}_t}} \big)dt 
			+ \sum_{l=1}^d \big( {\color{blue}\cancel{\mathcal{J}_\textbf{x}\textbf{b}_l(t)\boldsymbol{\delta}_t dw_t^l}} 
			+              {\color{darkgreen}\cancel{\mathcal{J}_\textbf{u}\textbf{b}_l(t)\textbf{v}_t dw_t^l }} ~\big)\\
			&\qquad=  \underbrace{\Big( \mathcal{J}_\textbf{x}\textbf{a}(t) \Phi_t  \: dt +  \sum_{l=1}^d \mathcal{J}_\textbf{x}\textbf{b}_l(t)\Phi_t \: dw_t^l \Big)\textbf{q}_t}_{d\Phi_t\textbf{q}_t}
			+ \underbrace{(\mathcal{J}_\textbf{u}\textbf{a}(t)\textbf{v}_t+ \Psi_t)\:dt 
			+ \sum_{l=1}^d \mathcal{J}_\textbf{u}\textbf{b}_l(t)\textbf{v}_t \:dw_t^l}_{\Phi_td\textbf{q}_t}\\
			&\qquad~~ + \underbrace{\sum_{l=1}^d \sum_{l'=1}^d \mathcal{J}_\textbf{x}\textbf{b}_l(t)\Phi_t \: \Phi_t^{-1}  \mathcal{J}_\textbf{u}\textbf{b}_{l'}(t) \textbf{v}_t \: dw_t^l \: dw_t^{l'} }_{d\Phi_t d\textbf{q}_t} \\
			&\qquad \overset{(a)}{=}  \Big( {\color{purple}\cancel{\mathcal{J}_\textbf{x}\textbf{a}(t) \boldsymbol{\delta}_t  \: dt}} 
			+  \sum_{l=1}^d {\color{blue}\cancel{ \mathcal{J}_\textbf{x}\textbf{b}_l(t)\boldsymbol{\delta}_t \: dw_t^l}} \Big)
			+ ( {\color{darkred}\cancel{\mathcal{J}_\textbf{u}\textbf{a}(t)\textbf{v}_t}}+ \Psi_t)\:dt 
			+ \sum_{l=1}^d {\color{darkgreen}\cancel{\mathcal{J}_\textbf{u}\textbf{b}_l(t)\textbf{v}_t \:dw_t^l}}\\
			&\qquad~~ + \sum_{l=1}^d \sum_{l'=1}^d \mathcal{J}_\textbf{x}\textbf{b}_l(t)  \mathcal{J}_\textbf{u}\textbf{b}_{l'}(t) \textbf{v}_t \: q_{l,l'} dt,
		\end{align*}
		where for (a), we used $dw_t^l dw_t^{l'}=q_{l,l'}dt$ and $\boldsymbol{\delta}_t = \Phi_t \textbf{q}_t$ based on \cref{Eq:constant-of-variation2}.
		By canceling the similar terms, we then obtain:
		$$
		\Psi_t =-\sum_{l,l'}^d q_{l,l'}\mathcal{J}_\textbf{x}\textbf{b}_l(t)  \mathcal{J}_\textbf{u}\textbf{b}_{l'}(t) \textbf{v}_t.
		$$
		Plugging the recent equation into \cref{Eq:Aux6},
		utilization of \cref{Eq:constant-of-variation2}
		and considering that $\textbf{q}_0 =\textbf{0}$ as $\boldsymbol{\delta}_t=\textbf{0}$ and $\Phi_0$ is invertible, thus gives:
		\begingroup\makeatletter\def\f@size{9.5}\check@mathfonts
		\begin{align}
			\label{Eq:delta_vector}
			\boldsymbol{\delta}_t = \Phi_t \textbf{q}_t = \int_0^t \Phi_t\Phi_s^{-1} \Big( \mathcal{J}_\textbf{u}\textbf{a}(s)
			-\sum_{l,l'}^d q_{l,l'}\mathcal{J}_\textbf{x}\textbf{b}_l(s)  \mathcal{J}_\textbf{u}\textbf{b}_{l'}(s) \Big)\:\textbf{v}_s\:ds 
			+\int_0^t \sum_{l=1}^d \Phi_t\Phi_s^{-1} \mathcal{J}_\textbf{u}\textbf{b}_l(s)\:\textbf{v}_s\:dw_s^l.
		\end{align}
		\endgroup
		As declared, $\Phi_t$ is equal to the stochastic flow $Y_t$~(\ref{Eq:Y-Z}).
		Hence, based on \cref{Eq:MatrixMalliavin}, we can get:
		$$
		\Phi_t\Phi_s^{-1}  = \boldsymbol{\Gamma}_{s,t},
		$$
		By substituting this into \cref{Eq:delta_vector}, the statement thus follows.
	\end{proof}
\end{lemma}

We now present \cref{Thm2} with a proof.
Note that without loss of generalization, we assume that $h(\textbf{x}_T) = 0$.
\begin{theorem}\label{Thm2_app}
	Consider SOCP $\mbox{P}_1$ (\ref{EQ:StochasticControl}),
	the variation of cost functional $J(\textbf{u},\textbf{x}_0)$
	along direction $\textbf{v}$, through Definition~(\ref{Def1}), can be obtained by:
	\begingroup\makeatletter\def\f@size{9.}\check@mathfonts
	\begingroup\makeatletter\def\f@size{9}\check@mathfonts
	\begin{align*}
		\delta_\textbf{v} J(\textbf{u},\textbf{x}_0) 
		= \mathbb{E} \!\int_0^T \! 
		\bigg\{ \!&
		\Big( \int_s^T \!  \nabla_\textbf{x}L(t)^\top \boldsymbol{\Gamma}_{s,t} \: dt+ \nabla_\textbf{x}h(\textbf{x}_T)^\top \boldsymbol{\Gamma}_{s,T} \Big) \Big( \mathcal{J}_\textbf{u}\textbf{a}(s)
		- \sum_{l,l'}^d q_{l,l'}\mathcal{J}_\textbf{x}\textbf{b}_l(s)  \mathcal{J}_\textbf{u}\textbf{b}_{l'}(s) \Big) \\
		&+ \sum_{l=1}^d 
		\Big(
		\int_s^T \!  \big({D_s^l\textbf{x}_t}^\top  \nabla_{\textbf{x}}^2L(t) + {D_s^l\textbf{u}_t}^\top  \nabla_{\textbf{u}\textbf{x}}L(t)\big) \boldsymbol{\Gamma}_{s,t} dt
		+ {D_s^l\textbf{x}_T}^\top \nabla_{\textbf{x}}^2h(\textbf{x}_T) \boldsymbol{\Gamma}_{s,T}
		\Big) 
		\:\mathcal{J}_\textbf{u}\textbf{b}_l(s) \\
		&+ \nabla_\textbf{u}L(s)^\top 
		\bigg\} \textbf{v}_s ds,
	\end{align*}
	\endgroup
	\endgroup
	where $\textbf{B} = [\textbf{b}_1,...,\textbf{b}_d]$,
	$\boldsymbol{\Gamma}_{s,t}$ relates to the Malliavin derivative $D_s\textbf{x}_t$ though \cref{Eq:MatrixMalliavin},
	$q_{l,l'}$ is the diffusion coefficient between Wiener processes $w_t^l$ and $w_t^{l'}$,
	$[\textbf{A}]_l$ outputs the $l$-th column of matrix $\textbf{A}$,
	and $\mathcal{J}_\textbf{x}\textbf{a}$ (respectively, $\mathcal{J}_\textbf{u}\textbf{a}$) stands for the Jacobian matrix of vector $\textbf{a}$ 
	w.r.t. $\textbf{x}$ (respectively, $\textbf{u}$).
	
	\begin{proof}
		We perturb the control decision $\textbf{u}_t$ through $\textbf{u}_t+  \epsilon \textbf{v}_t$, 
		which leads to the diffusion process $\textbf{x}_t^{\textbf{u}_t+  \epsilon \textbf{v}_t}$.
		We then define a variational process:
		$$
		\boldsymbol{\delta}_t^x := \lim_{\epsilon \to 0}\frac{1}{\epsilon}(\textbf{x}_t^{\textbf{u}+  \epsilon \textbf{v}} - \textbf{x}_t^{\textbf{u}}).
		$$
		Based on \cref{Eq:SDE}, $\boldsymbol{\delta}_t^x $ obeys the following variational SDE:
		\begin{align}\label{Eq:vec-variationa_x}
		d\boldsymbol{\delta}_t^x = \big( \mathcal{J}_\textbf{x}\textbf{a}(t)\boldsymbol{\delta}_t^x 
		+  \mathcal{J}_\textbf{u}\textbf{a}(t)\textbf{v}_t \big)dt 
		+ \sum_{l=1}^d\big( \mathcal{J}_\textbf{x}\textbf{b}_l(t)\boldsymbol{\delta}_t^x +  \mathcal{J}_\textbf{u}\textbf{b}_l(t)\textbf{v}_t \big)dw_t^l.
		\end{align}
		Using Lemma (\ref{Lemma2}), the solution of this variational SDE is:
		\begin{align}\label{Eq:vec-solution-delta_x}
			\boldsymbol{\delta}_t^x = \int_0^t \boldsymbol{\Gamma}_{s,t} 
			\Big( \mathcal{J}_\textbf{u}\textbf{a}(s)- \sum_{l,l'}^d q_{l,l'}\mathcal{J}_\textbf{x}\textbf{b}_l(s)  \mathcal{J}_\textbf{u}\textbf{b}_{l'}(s)\Big)\:\textbf{v}_s\:ds 
			+\int_0^t \sum_{l=1}^d \boldsymbol{\Gamma}_{s,t} \mathcal{J}_\textbf{u}\textbf{b}_l(s)\:\textbf{v}_s\:dw_s^l.
		\end{align}
		On the other hand, the variation of cost functional $J(\textbf{u},\textbf{x}_0)$ along direction $\textbf{v}$ obeys:
		\begin{align}\label{Eq:vec-variation-of-cost}
			\delta_\textbf{v}J(\textbf{u},\textbf{x}_0) = \mathbb{E}\Big\{ \int_0^T  \big(\nabla_\textbf{x}L(t)^\top \boldsymbol{\delta}_t^x + \nabla_\textbf{u}L(t) \textbf{v}_t\big) dt\Big\} .
		\end{align}
		By plugging \cref{Eq:vec-solution-delta_x} into \cref{Eq:vec-variation-of-cost}, we then have:
		\begin{align*}
			\delta_\textbf{v}J(\textbf{u},\textbf{x}_0) =& \mathbb{E} \int_0^T \nabla_\textbf{x}L(t)^\top \int_0^t \boldsymbol{\Gamma}_{s,t} \Big( \mathcal{J}_\textbf{u}\textbf{a}(s)- \sum_{l,l'}^d q_{l,l'}\mathcal{J}_\textbf{x}\textbf{b}_l(s)  \mathcal{J}_\textbf{u}\textbf{b}_{l'}(s)\Big) \textbf{v}_s ds \: dt \\
			&+\mathbb{E} \int_0^T \sum_{l=1}^d \nabla_{\textbf{x}}L(t)^\top \int_0^t \boldsymbol{\Gamma}_{s,t} \mathcal{J}_\textbf{u}\textbf{b}_l(s) \textbf{v}_s dw_s^l \: dt			      
			+\mathbb{E} \int_0^T \mathcal{J}_\textbf{u}L(t)^\top  \textbf{v}_t dt.
		\end{align*}
		We then apply the integration-by-parts Lemma~(\ref{lemma:integation-by-parts}) for the second term of RHS of the recent equation:
		\begin{align*}
			\delta_\textbf{v}J(\textbf{u},\textbf{x}_0) =& \mathbb{E} \int_0^T \int_0^t \nabla_\textbf{x}L(t)^\top \boldsymbol{\Gamma}_{s,t} \Big( \mathcal{J}_\textbf{u}\textbf{a}(s)- \sum_{l,l'}^d q_{l,l'}\mathcal{J}_\textbf{x}\textbf{b}_l(s)  \mathcal{J}_\textbf{u}\textbf{b}_{l'}(s)\Big) \:\textbf{v}_s \: ds \: dt \\
			&+\mathbb{E} \int_0^T \int_0^t \sum_{l=1}^d 
			\big(
			{D_s^l\textbf{x}_t}^\top \nabla_{\textbf{x}}^2L(t) + {D_s^l\textbf{u}_t}^\top \nabla_{\textbf{u}\textbf{x}}L(t)
			\big) \boldsymbol{\Gamma}_{s,t} \mathcal{J}_\textbf{u}\textbf{b}_l(s) \: \textbf{v}_s \: ds \: dt \\
			&+\mathbb{E} \int_0^T \mathcal{J}_\textbf{u}L(t)^\top  \textbf{v}_t dt,
		\end{align*}
		where $D_s^l\textbf{x}_t$ (respectively, $D_s^l\textbf{u}_t$) is the Malliavin derivative of $\textbf{x}_t$ (respectively, $\textbf{u}_t$) 
		against $w_t^l$.		
		Changing the order of integrals then gives:
		\begin{align*}
			\delta_\textbf{v}J(\textbf{u},\textbf{x}_0) =& \mathbb{E} \int_0^T \int_s^T \nabla_\textbf{x}L(t)^\top \boldsymbol{\Gamma}_{s,t} \Big( \mathcal{J}_\textbf{u}\textbf{a}(s)- \sum_{l,l'}^d q_{l,l'}\mathcal{J}_\textbf{x}\textbf{b}_l(s)  \mathcal{J}_\textbf{u}\textbf{b}_{l'}(s)\Big)  \: dt \: \textbf{v}_s \: ds \\
			&+\mathbb{E} \int_0^T \int_s^T \sum_{l=1}^d 
			\big(
			{D_s^l\textbf{x}_t}^\top \nabla_{\textbf{x}}^2L(t) + {D_s^l\textbf{u}_t}^\top \nabla_{\textbf{u}\textbf{x}}L(t)
			\big) \boldsymbol{\Gamma}_{s,t} \mathcal{J}_\textbf{u}\textbf{b}_l(s)  \: dt \: \textbf{v}_s \: ds \\
			&+\mathbb{E} \int_0^T \mathcal{J}_\textbf{u}L(t)^\top  \textbf{v}_t dt.
		\end{align*}
		According to \cref{Eq:MatrixMalliavin}, we have $D_s^l\textbf{x}_t = [\boldsymbol{\Gamma}_{s,t}\textbf{B}(s)]_l$.
		This thus proves statement.
	\end{proof}
\end{theorem}

\section{proof of Proposition \ref{Propos:BackwardSDE}}\label{App3}
 
According to \cref{Eq:MalliavinDerv0}, the Malliavin derivative $D_r\textbf{x}_t$ -- \emph{for a given $r$ and as a function of $t$} -- is a diffusion process adapted to the filtration $\mathcal{F}_r$. 
Therefore, based on \cref{Eq:MatrixMalliavin}, the process $\boldsymbol{\Gamma}_{r,t} := \textbf{Y}_t \textbf{Z}_r$, for a given $r$ is adapted to $\mathcal{F}_r$.
On the other hand, according to \cref{Eq:Y-Z} we can obtain
\begin{align*}
d \boldsymbol{\Gamma}_{r,t} &= d\:\textbf{Y}_t \: \textbf{Z}_r 
= \mathcal{J}_\textbf{x}\textbf{a}(\textbf{x}_t,\textbf{u}_t)\textbf{Y}_t \textbf{Z}_r dt
+ \sum_{l=1}^d \mathcal{J}_\textbf{x}\textbf{b}_l(\textbf{x}_t,\textbf{u}_t)\textbf{Y}_t\textbf{Z}_r dw_t^l \\
&= \Bigg( \sum_{l=1}^d \mathcal{J}_{\textbf{x}}\textbf{b}_l(\textbf{x}_t,\textbf{u}_t)\: dw_t^l 
+ \mathcal{J}_{\textbf{x}}\textbf{a}(\textbf{x}_t,\textbf{u}_t) dt  \Bigg)\boldsymbol{\Gamma}_{r,t}.
\end{align*}
Therefore, statement follows.

\end{document}